\documentclass[11pt]{article}
\usepackage[utf8]{inputenc} 
\usepackage{latexsym,amsmath,amssymb,textcomp}
\usepackage[all]{xy}
\usepackage[nottoc, notlof, notlot]{tocbibind}
\usepackage{geometry} 
\usepackage{graphicx} 
\usepackage{amsthm}
\usepackage{amsmath}
\usepackage{booktabs}
\usepackage{array} 
\usepackage{paralist}
\usepackage{verbatim} 
\usepackage{subfig} 
\usepackage{hyperref}
\usepackage{fancyhdr} 
\title{Around evaluations of biset functors.}
\author{Baptiste Rognerud}
\begin{document}
\maketitle
\theoremstyle{plain}
\newtheorem{theo}{Theorem}[section]
\newtheorem{theox}{Theorem}[section]
\renewcommand{\thetheox}{\Alph{theox}}
\newtheorem{prop}[theo]{Proposition}
\newtheorem{lemma}[theo]{Lemma}
\newtheorem{coro}[theo]{Corollary}
\newtheorem{question}[theo]{Question}
\newtheorem*{notations}{Notations}
\theoremstyle{definition}
\newtheorem{de}[theo]{Definition}
\newtheorem{dex}[theox]{Definition}
\theoremstyle{remark}
\newtheorem{ex}[theo]{Example}
\newtheorem{re}[theo]{Remark}
\newtheorem{res}[theo]{Remarks}
\renewcommand{\labelitemi}{$\bullet$}
\newcommand{\comu}{co\mu}
\newcommand{\mo}{\ensuremath{\hbox{mod}}}
\newcommand{\Mo}{\ensuremath{\hbox{Mod}}}
\newcommand{\decale}[1]{\raisebox{-2ex}{$#1$}}
\newcommand{\decaleb}[2]{\raisebox{#1}{$#2$}}
\newcommand\Ind{{\rm Ind}}
\newcommand\Inf{{\rm Inf}}
\newcommand\Res{{\rm Res}}
\newcommand\Def{{\rm Def}}
\newcommand\Iso{{\rm Iso}}
\newcommand\field {{k}}
\newcommand\Mod {\hbox{-}{\rm Mod}}
\newcommand\Hom {{\rm Hom}}
\newcommand\End {{\rm End}}
\newcommand\Out{{\rm Out}}
\newcommand\mb{\hbox{-}}
\newcommand\D{\mathcal{D}}
\newcommand\C{\mathcal{C}}
\newcommand\FD{\mathcal{F}_{\mathcal{D},R}}
\newcommand\FDp{\mathcal{F}_{\mathcal{D}',R}}
\newcommand\FDk{\mathcal{F}_{\mathcal{D},k}}
\newcommand\FDR{\mathcal{F}_{\mathcal{D},R}}
\newcommand\FSk{\mathcal{F}_{G,k}}
\newcommand\Rad{{\rm Rad}}
\newcommand\Top{{\rm Top}}
\newcommand\kbg{kB(G,G)}
\newcommand\NV{\mathrm{NV}}
\begin{abstract}
The evaluation functor carries important information about the category of biset functors into the category of modules over the double Burnside algebra. Our purpose here, is to study double Burnside algebras via evaluations of biset functors. In order to avoid the difficult problem of vanishing of simple functors, we look at finite groups for which there is no non-trivial vanishing and we call them non-vanishing groups. This family contains all the abelian groups, but also infinitely many others. We show that for a non-vanishing group, there is an equivalence between the category of modules over the double Burnside algebra and a specific category of biset functors. Then, we deduce results about the highest-weight structure, and the self-injective property of the double Burnside algebra. We also revisit Barker's Theorem on the semi-simplicity of the category of biset functors. 
\end{abstract}
\par\noindent
{\it {\footnotesize A.M.S. subject classification: 19A22, 20C99, 16G10,18E10.}}
\section{Introduction}
Let $k$ be a field and $G$ be a finite group. The double Burnside ring $B(G,G)$ is the Grothendieck ring of the category of all finite $G$-$G$-bisets. Extending scalars to $k$, we have an algebra $kB(G,G)$ which is called the double Burnside algebra of $G$. The elements of this algebra are the formal linear combinations of $G$-$G$-bisets and the product behave like a tensor product. 

This algebra, and some of its subalgebras are crucial objects in some recent developments in the modular representation theory of finite groups, fusion systems and homotopy theory. We refer to the introduction of \cite{bst1}, for explicit motivations. In the last fifteen years, it has been studied by several mathematicians under several points of view. In \cite{bd1}, Robert Boltje and Susanne Danz were particularly interested by the subalgebras consisting of left-free bisets or bi-free bisets in characteristic zero. In \cite{bk}, Boltje and Burkhard Külshammer found the central primitive idempotents of these two algebras. However, the question is still open for the double Burnside algebra, even in characteristic zero. 

More generally, it is particularly difficult to generalize the results of the left-free double Burnside algebra to the whole double Burnside algebra. According to Jacques Thévenaz, there is a `quantic gap' between them. 

Using a completely different approach, it was recently shown by Serge Bouc, Radu Stancu and Jacques  Thévenaz in \cite{bst1}, that one can deduce a lot of information about the double Burnside algebra via evaluating biset functors. More surprisingly, they showed that the converse is also true: one can deduce a lot about biset functors by looking at double Burnside algebras using adjoints of the evaluation functor, but also more sophisticated tools. The goal of this article is to continue to develop and use this philosophy in order to have a better understanding of the double Burnside algebra.

With this approach, we are immediately stuck with the problem of evaluating simple biset functors: it is almost obvious that the evaluation at $G$ of a simple biset functor $S$ is either $0$ or a simple module over the endomorphism algebra of $G$. Unfortunately, it is a notoriously difficult combinatorial problem to understand when there is a vanishing. With our approach, we observe that this problem is not only annoying, it is also crucial when one wants to understand the double Burnside algebra of a finite group $G$ (See for example Corollary \ref{lift_proj}). We also remark that some problems of Section $9$ of \cite{bst1} are nothing but problems of vanishing of some simple functors. 

The simple functors are indexed by a \emph{minimal group} $H$ and a simple $k\mathrm{Out}(H)$-module. If $H$ is not isomorphic to a subquotient of $G$, then it is clear that the corresponding simple functor vanishes at $G$. This is what we call a \emph{trivial vanishing}. However, when $H$ is a subquotient of $G$ it is still possible that the simple functor vanishes at $G$ (see Corollary \ref{ex_vanishing}). In this case we speak about a \emph{non-trivial} vanishing. 

In order to avoid this difficulty, we consider finite groups such that there are no non-trivial vanishing of simple functors. We show that it is possible to reformulate this condition, involving the simple functors in an elementary condition, involving a composition of bisets. We call it the \emph{generating} relation. We give various equivalent interpretations of this relation. One in terms of vanishing of biset functors, another in terms of a composition of bisets and finally one in terms of representable functors in the biset category. As the main result, we show that the non-vanishing condition implies the existence of an \emph{equivalence of categories} between the category of modules over the double Burnside algebra of a finite group $G$ and the category of biset functors over $G$.

This gives us the desire to understand the generating relation and, more generally, the family of \emph{non-vanishing groups}. We show that the \emph{abelian groups} and the so-called self-dual groups are non-vanishing. Unfortunately, we did not succeed to classify all the non-vanishing groups. Understanding this family is in theory much easier than understanding the vanishing problem for the simple biset functors: it is enough to understand the generating relation. At first it seems easy, but we think that it is still a difficult problem. At the end of Section $5$, we give an illuminating example of a non-vanishing group in characteristic $0$. It can be seen that the non-vanishing property of this group is connected to some non-trivial facts about the ordinary representations of the group $GL(3,2)$. This example is so  pathological that it seems very unlikely to end up with a classification. 

In the rest of the article, we use the evaluation functor and the generating relation in order to deduce new information about the ring structure of the double Burnside algebra.

In Section $6$, we assume the field to be of characteristic zero. We show that the double Burnside algebra of a non-vanishing group is \emph{quasi-hereditary}. This result is nothing but an application of the equivalence discussed above and the famous Theorem of Peter Webb about the highest-weight structure of the category of biset functors. Quasi-hereditary algebras, and highest-weight categories come historically from the theory of representations of the complex semi-simple Lie algebras. A lot of important notions for Lie algebras admit a generalization, or an axiomatization, to quasi-hereditary algebras. In particular, there is a notion of \emph{exact Borel subalgebra} of a given quasi-hereditary algebra. These subalgebras seem to be of first importance in the recent development of the representation theory of algebras.  They are particularly important for the so-called theory of \emph{bocses}. It is known, and highly non-trivial, that for any quasi-hereditary $A$, there is a Morita equivalent algebra $A'$ which admits an exact Borel subalgebra. For biset functors, it turns out that this notion of Borel subalgebras is connected with the so-called \emph{deflation} functors. That is, biset functors without \emph{inflation}. We show that the category of deflation functors is an exact Borel subcategory of the category of biset functors. In order to prove this result, we observe that the induction functor between the category of deflation functors and the full category of biset functors is \emph{exact} without any assumption on the field or on the ring of coefficients.

For the double Burnside algebra, the situation is as usual more complicated. This algebra may be an example of a quasi-hereditary algebra which does not admit an exact Borel subalgebra. First, we show that over a field of characteristic zero, the left-free double Burnside algebra is a quasi-hereditary algebra. This result may be of independent interest. Nevertheless, the left-free double Burnside algebra is not in general a Borel subalgebra of the double Burnside algebra. The main reason is that there are less simple modules over the left-free algebra than over the whole double Burnside algebra. 

In the section $7$, we use the \emph{generating} relation to \emph{revisit} Barker's Theorem about the semi-simple property of the category of biset functors. We show, that understanding the semi-simplicity of the double Burnside algebra is equivalent to understand the semi-simplicity of the category of biset functors. We also deduce a useful characterization of this semi-simplicity. It is well-known that a group algebra is semi-simple if and only if the trivial module is projective. We prove that for biset functors and modules over the double Burnside algebra, we have a similar result for a suitable notion of trivial object. 

For $A_5$, one can see that the double Burnside algebra has infinite global dimension. More precisely, one can show that it has a self-injective block. For that reason, we feel natural to study the self-injective property of the double Burnside algebra. Using the characterization of the semi-simplicity by the trivial object, we show that the double Burnside algebra of a finite group is never a self-injective algebra except when it is semi-simple. 
 \newpage
 \tableofcontents
\section{Review on bisets and biset functors}
In this short section, we fix our notation and for the convenience of the reader, we recall some well known facts about biset functors that are crucial for the present article. We refer to the first chapters of \cite{bouc_biset} for more details. 
\newline\indent Let $G$ be a finite group. We denote by $s_G$ the set of the subgroups of $G$ and by $[s_G]$ a set of representatives of the conjugacy classes of the subgroups of $G$. As always, if $g\in G$ and $H$ is a subgroup of $G$, then we denote by $\, ^{g}H$ the conjugate of $H$ by $g$. 
\newline\indent Let $G$ and $H$ be two finite groups. Let $L$ be a subgroup of $G\times H$. There are four important subgroups associated to $L$:  
\begin{align*}
p_1(L) &= \{ g\in G\ ;\ \exists h\in H,\ (g,h)\in L\},\\
p_2(L) &= \{ h\in H\ ;\ \exists g\in G,\ (g,h)\in L\},\\
k_1(L) & = \{ g\in G\ ;\ (g,1)\in L\},\\
k_2(L) &= \{ h\in H\ ;\ (1,h)\in L\}. 
\end{align*}
It is clear that $k_i(L) \trianglelefteq p_i(L)$ for $i=1,2$ and that $\Big(k_1(L)\times k_2(L)\Big)\trianglelefteq L$. Moreover, there are canonical isomorphisms $$p_1(L)/k_1(L) \cong L/\Big(k_1(L)\times k_2(L)\Big) \cong p_2(L)/k_2(L).$$ The quotient $L/\Big(k_1(L)\times k_2(L)\Big)$ will be denoted by $q(L)$. 
\newline\indent Let $G$, $H$ and $K$ be three finite groups. Let $L$ be a subgroup of $G\times H$ and $M$ be a subgroup of $H\times K$. Then $L\star M$ is the subgroup of $G\times K$ defined by:
$$L\star M := \{ (g,k)\in G\times K\ ;\ \exists h\in H,\ (g,h)\in L \hbox{ and } (h,k)\in M \}. $$
Let $G$ and $H$ be two finite groups. A $G$-$H$-biset is a set endowed with a left action of $G$ and a right action of $H$ which commute. In other terms, A $G$-$H$-biset is nothing but a $G\times H^{op}$-set. If $L$ is a subgroup of $G\times H$, the quotient $(G\times H)/L$ is naturally a $G$-$H$-biset for the action given by
\[ a\cdot (g,h)L \cdot b = (ag,b^{-1}h)L, \forall a,g\in G, \forall b,h\in H.\]
The \emph{double Burnside} group $B(G,H)$ is the Grothendieck group of the category of finite $G$-$H$-bisets. The set of $[(G\times H)/L ]$ where $L\in [s_{G\times H}]$ is called the canonical basis of $B(G,H)$. Here, $[X]$ denotes the isomorphism class of the $G$-$H$-biset $X$. 
\newline\indent Let $G$, $H$ and $K$ be three finite groups. Let $U$ be a $G$-$H$-biset and $V$ be an $H$-$K$-biset. Then, we denote by $U\times_{H} V$ the set of $H$-orbits of $U\times V$ where $H$ acts diagonally on the cartesian product. That is $h\cdot (u,v) = (uh^{-1},hv)$ for $h\in H$ and $(u,v)\in U\times V$.  Extending this by bilinearity, we have a bilinear map from $B(G,H)\times B(H,K)$ to $B(G,K)$. This product will be understood as a composition, so we will sometimes use the notation $\circ$ instead of $\times_{H}$.
\newline\indent There is a Mackey formula for the composition of the transitive bisets. Let $G$, $H$ and $K$ be three finite groups. Let $L$ be a subgroup of $G\times H$ and $M$ be a subgroup of $H\times K$. Then, by Lemma $2.3.24$ of \cite{bouc_biset}, we have
\begin{equation}\label{mackey}
\big((G\times H)/L\big)\times_{H} \big((H\times K)/M\big) \cong \bigsqcup_{h\in [p_{2}(L)\backslash H /p_{1}(M) ]} (G\times K) / \big( L\star\, ^{(h,1)}M\big)
\end{equation}
\newline\indent Let $R$ be a commutative ring with $1$. For $G$ and $H$ two finite groups, we denote by $RB(G,H)$ the $R$-module $R\otimes_{\mathbb{Z}} B(G,H)$. We still denote by $\times_H$ its $R$-bilinear extension.
\begin{de}
The biset category $R\mathcal{C}$ over $R$ is the category where:
\begin{itemize}
\item The objects are the finite groups.
\item If $G$ and $H$ are two finite groups, then $\Hom_{R\C}(G,H)=RB(H,G)$. 
\item If $G$, $H$ and $K$ are finite groups, the composition is the product
\[ -\times_H - : RB(G,H)\times RB(H,K) \to RB(G,K). \] 
\item If $G$ is a finite group, then the identity morphism is $[(G\times G)/ \Delta(G)]$ where $\Delta(G)$ is the diagonal subgroup of $G$. 
\end{itemize}
A \emph{biset functor} over $R$ is an $R$-linear functor from $R\C$ to $R\Mod$. 
\end{de}
Here, we choose to apply an `op'-functor on the set of morphisms. This is just for convenience and this will allow us to work with covariant functors instead of contravariant functors. 
\newline\indent The category of bisets has a very important property. Every morphism can be written as sum of transitive morphisms. Moreover every transitive morphism can be \emph{factorised} and written as a composition of $5$ particular morphisms which are called \emph{elementary} by Bouc. This is Lemma $2.3.26$ of \cite{bouc_biset}. If $H$ is a subgroup of $G$, then we set $\Ind_{H}^{G} := \,_{G}G_H$ and $\Res^{G}_{H} = \, _{H}G_G$, where the action of $G$ and $H$ is given by the multiplication. If $N$ is a normal subgroup of $G$, then we set $\Inf_{G/N}^{G} = \, _{G} (G/N)_{G/N}$ and $\Def_{G/N}^{G} = \, _{G/N}(G/N)_{G}$ where $G/N$ acts via multiplication and $G$ via the canonical projection onto $G/N$. If $\alpha : H\to H'$ is a group isomorphism, then we set $\Iso(\alpha) = \, _{H'}H'_{H}$ where $H$ acts via the morphism $f$. Then, we have Bouc's \emph{Butterfly Lemma} (Lemma $2.3.26$ \cite{bouc_biset}). Let $G$ and $H$ be two finite groups. Let $L$ be a subgroup of $G\times H$. Then,
\begin{equation}\label{butterfly}
\big(G\times H\big)/L \cong  \Ind_{p_1(L)}^{G} \circ \Inf_{p_1(L)/k_{1}(L)}^{p_1(L)}\circ \Iso(\alpha) \circ \Def^{p_2(L)}_{p_2(L)/k_2(L)}\circ \Res^{H}_{p_2(L)},
\end{equation}
where $\alpha$ is the canonical isomorphism between $p_{2}(L)/k_2(L)$ and $p_{1}(L)/k_{1}(L)$. 
\newline\indent We are particularly interested by subcategories of the biset category in which the morphisms will still have similar decompositions. These categories are called admissible by Bouc.
\begin{de}[Definition $4.1.3$ \cite{bouc_biset}]
A subcategory $\D$ of $R\mathcal{C}$ is called admissible if it contains group isomorphisms and if it satisfies the following conditions:
\begin{itemize}
\item If $G$ and $H$ are objects of $\D$, then there is a subset $S(H,G)$ of the set of subgroups of $H\times G$, invariant under $(H\times G)$-conjugation, such that $\Hom_{\D}(G,H)$ is the submodule of $RB(H,G)$ generated by the elements $[(H\times G)/L]$, for $L\in S(H,G)$.
\item If $G$ and $H$ are groups of $\D$, and if $L\in S(H,G)$, then $q(L)$ is also an object of $\D$. Moreover, $\Def^{p_2(L)}_{p_2(L)/k_{2}(L)} \circ \Res^{G}_{p_2(L)}$ and $\Ind_{p_1(L)}^{H}\circ \Inf_{p_1{L}/k_{1}(L)}^{p_1(L)}$ are morphisms in $\D$. 
\end{itemize}
\end{de}
If $\D$ is an admissible biset category and $H,K \in \D$, we denote by $\D(K,H)$ the set $\Hom_{\D}(H,K)$. 
\newline\indent In this article we are mostly interested by two types of admissible biset categories: \emph{replete} biset categories and their subcategories consisting of the \emph{left-free bisets}. More precisely: let $G$ be a finite group. A \emph{section} of $G$ is a pair $(B,A)$ such that $A \trianglelefteq B \leqslant G$. The quotient $B/A$ is then called a \emph{subquotient} of $G$. If $H$ is isomorphic to a subquotient of $G$, then we use the notation $H\sqsubseteq G$. If it is isomorphic to a strict subquotient, we use the notation $H\sqsubset G$. We denote by $\Sigma(G)$ the set of all subquotients of $G$. Moreover if $\D$ is a class of finite groups, we say that $\D$ is closed under taking subquotients if whenever $H$ is a group isomorphic to $B/A$ where $(B,A)$ is a section of a group $G$ in $\D$, then $H$ is also in $\D$.
\begin{de}
Let $R$ be a commutative ring. A subcategory of $R\C$ is called \emph{replete} if it is a full subcategory of $R\C$ whose class of objects is closed under taking subquotients. 
\end{de}
In this article, we will manly work with replete biset categories and their subcategories consisting of left-free bisets. If $G$ is a finite group, then we abusively denote by $\Sigma(G)$ the full subcategory of $R\mathcal{C}$ consisting of all the groups isomorphic to subquotients of $G$. It is clearly a replete biset category. 
\newline\indent If $\D$ is an admissible biset category, then we denote by $\mathcal{F}_{\mathcal{D},R}$ the category of $R$-linear functors from $\D$ to $R\Mod$, and we call it the category of \emph{$R$-biset functors over $\D$}. 
\begin{de}
Let $R$ be a commutative ring and $G$ be a finite group. The category $\mathcal{F}_{\Sigma(G),R}$ is called the category of $R$-biset functors over $G$ and simply denoted by $\mathcal{F}_{G,R}$. 
\end{de} 
\begin{de}\label{double_burnside}
Let $R$ be a commutative ring with $1$ and $G$ be a finite group. The double Burnside algebra of $G$ is the endomorphism algebra of $G$ in $R\C$. In other words, the double Burnside algebra of $G$ is the module $RB(G,G)$ endowed with the product induced by the composition of $G$-$G$-bisets. 
\end{de}
Let $\D$ be an admissible biset category. As the category $R\Mod$ is abelian, it is well known that the category of biset functors over $\D$ is an \emph{abelian category}. The abelian structure is point-wise. In other words, it is defined on the evaluations of the functors. 
\newline\indent By Yoneda's Lemma, if $G$ is an object of $\D$, the \emph{representable functor} (also called the Yoneda functor) $Y_G:=\Hom_\D(G,-)=\D(-,G)$ is a projective object of $\FDR$. Moreover, every object of $\FDR$ is a quotient of a direct sum of Yoneda functors. So, the category $\FDR$ has \emph{enough projective} objects. Using a duality argument one can show that this category has \emph{enough injective} objects. See Corollary 3.2.13 of \cite{bouc_biset} for more details. We are also particularly interested by the family of simple functors which can be described (see the next paragraph) using the so-called evaluation functor.
\newline\indent Let $G \in \D$. If $F$ is a biset functor over $\D$, then its value $F(G)$ has a natural structure of module over the endomorphism algebra $\D(G,G)$ of $G$.  More precisely, we have a functor $ev_G : \FDR\to  \D(G,G)\Mod$. Since the abelian structure of $\FDR$ is defined on the evaluations, this functor is clearly exact. By usual arguments, it has a left and a right adjoint. A left adjoint denoted by $L_{G,-}$ can be defined as follows. Let $H \in \D$. Then, the right-multiplication by the elements of $\D(G,G)$ induces a structure of right $\D(G,G)$-module on $\D(H,G)$. Let $V$ be a $\D(G,G)$-module. Then, we set $L_{G,V}(H) := \D(H,G)\otimes_{\D(G,G)} V$. It is now straightforward to check that $L_{G,V}:= \D(-,G)\otimes_{\D(G,G)} V$ is a biset functor over $\D$ and that $V\mapsto L_{G,V}$ is a functor from $\D(G,G)\Mod$ to $\FDR$ (for more details and proofs see Section $3.3$ of \cite{bouc_biset}). 

If $V$ is a simple $\D(G,G)$-module, then $L_{G,V}$ has a unique simple quotient. Moreover it is easy to see that any simple functor appears as such quotient. However, this construction is not completely satisfying as one simple functor may be realized by many different pairs $(G,V)$. One can avoid this problem by considering minimal groups and simple modules over a quotient of the double Burnside algebra.

By Proposition $4.3.2$ of \cite{bouc_biset}, the quotient of the algebra $\D(G,G)$ by the ideal $I_{\D}(G,G)$ consisting of all the morphisms factorizing through groups strictly smaller than $G$ is isomorphic to the algebra of outer automorphisms of the group $H$, denoted by $R\Out(H)$. 
\begin{theo}
Let $k$ be a field. Let $\D$ be an admissible biset category. The set of isomorphism classes of simple objects of $\FDk$ is in bijection with the set of isomorphism classes of pairs $(H,V)$ where $H$ runs through the objects of $\D$ and $V$ through the simple $k\Out(H)$-modules. 
\end{theo}
\begin{proof}
See Theorem $4.3.10$ of \cite{bouc_biset}. 
\end{proof}
If $(H,V)$ is a pair consisting of a finite group and a simple $k\Out(H)$-module, then since $k\Out(H)$ is a quotient of $\D(H,H)$, one can see $V$ as a simple $\D(H,H)$-module by inflation. Then $S_{H,V}$ is nothing but the quotient $L_{H,V}/J_{H,V}$, where $J_{H,V}$ is the unique maximal subfunctor of $L_{H,V}$. If $K \in \D$, then by Remark $4.2.6$ of \cite{bouc_biset}, we have
\begin{equation}\label{jhv}
J_{H,V}(K) = \{ \sum_{i=1}^{n} \phi_i \otimes v_i \in L_{H,V}(K)\ ;\ \forall \psi\in \D(H,K),\ \sum_{i=1}^{n}(\psi \circ \phi_i)\cdot v_i=0\}.
\end{equation}
Note that this subfunctor has the property of vanishing at the group $H$. 
\newline\indent Let us recall that a biset functor is \emph{finitely generated} if it is a quotient of a finite direct sum of Yoneda functors. In particular, the simple functors and the Yoneda functors are finitely generated. As in the case of modules over a ring, the choice axiom implies the existence of maximal subfunctors of a finitely generated functor. As always, we define the radical of a biset functor $F$ as the intersection of all maximal subfunctors, and we denote it by $\Rad(F)$. Over a field, the category of finitely generated biset functors is a Krull-Schmidt category. In particular, every simple functor $S_{H,V}$ has a \emph{projective cover} in $\FDk$. We denote by $P_{H,V}$ a projective cover of $S_{H,V}$. 
\newline\indent In this article, we will also need the following family of biset functors. Let $H$ and $K$ be two objects of $\mathcal{D}$. Then $$\sum_{\underset{X\sqsubset H}{X\in \mathcal{D}}} \D(K,X)\D(X,H),$$ can be viewed as a submodule of $\D(K,H)$ via composition of morphisms. We denote by $I_{\mathcal{D}}(K,H)$ this submodule and by $k\overrightarrow{\D}(K,H)$ the quotient $k\D(K,H)/I_\mathcal{D}(K,H)$. This is a natural right $k\mathrm{Out}(H)$-module. If $V$ is a $k\mathrm{Out}(H)$-module, then we denote by $\Delta_{H,V}^{\mathcal{D}}$ the functor
\begin{equation*}
\Delta_{H,V}^{\mathcal{D}} := K\mapsto \overrightarrow{\D}(K,H)\otimes_{k\mathrm{Out}(H)}V. 
\end{equation*}
When the context is clear enough, we will simply denote it by $\Delta_{H,V}$. These functors were introduced first, in a different form, by Peter Webb in \cite{webb_strat2}. Webb proved that under suitable hypothesis they are the standard functors in the \emph{highest-weight} structure of the category of bisets functors. Since we think that this is a fundamental result, we will keep this idea in mind by calling them the \emph{standard biset functors}. Note that the module $V$ is not supposed to be simple here.  
\newline\indent Similarly,  we let $\overleftarrow{\D}(H,K)$ be the quotient of $\D(H,K)$ by the $R$-submodule consisting of all morphisms factorizing through groups strictly smaller than $H$. Let $V$ be a $k\Out(H)$-module. We denote by $\nabla_{H,V}^{\D}$ the functor
$$\nabla_{H,V}^{\D} := K\mapsto \Hom_{k\Out(H)}\big(\overleftarrow{\D}(H,K),V\big). $$ 
These functors are called the \emph{co-standard} biset functors. 
If $H\in \D$, then we have a functor from $\FDk$ to $k\Out(H)\Mod$ defined by $$F\mapsto \underline{F}(H) = \bigcap_{\underset{U\in \D(K,H)}{K\sqsubset H}} \mathrm{Ker}\big(F(U)\big).$$ 
\begin{lemma}
Let $\D$ be an admissible biset category. Let $H\in \D$. The functor sending a $k\Out(H)$-module $V$ to $\Delta_{H,V}$ is a left adjoint to the functor sending a biset functor $F$ to $\underline{F}(H)$.
\end{lemma}
\begin{proof}
This is straightforward. 
\end{proof}
\section{Evaluations of biset functors}
Let $k$ be a field and $\D$ be an admissible biset category. Let $G\in \D$. Then, we have an evaluation functor from $\FDk$ to $\D(G,G)\Mod$. It is well known that this evaluation functor carries a lot of informations of the category of biset functors into the category of $\D(G,G)$-modules. More recently, Bouc, Stancu and Thévenaz showed in \cite{bst1} and \cite{bst2} that the converse is also true. For example, we have:   
\begin{prop}\label{ev_fact}
Let $\D$ be an admissible biset category. Let $S$ be a simple biset functor over $\D$. Let $G$ be a finite group such that $S(G)\neq 0$. Let $F$ be a biset functor over $\D$. Then, the following are equivalent.
\begin{enumerate}
\item $S$ is isomorphic to a subquotient of $F$.
\item The simple $\D(G,G)$-module $S(G)$ is isomorphic to a subquotient of $F(G)$.
\end{enumerate}
\end{prop}
\begin{proof}
It is Proposition $3.5$ of \cite{bst1} for the case of a replete biset category. The proof is formal, it can be applied to any admissible biset category. 
\end{proof}
However, the evaluation at $G$ is not always compatible with algebraic operations. For example, it does not commute with taking the radical. Indeed, as explained in Section $9$ of \cite{bst1}, the evaluation at a group $G$ of the radical of a biset functor is not always the radical of the evaluation. In \cite{qh_doubleburnside}, we observed that this phenomenon is connected with the \emph{vanishing property} of the simple biset functors. 
\begin{lemma}\label{ev}
Let $\D$ be an admissible biset category. Let $F\in \mathcal{F}_{\mathcal{D},k}$ be a finitely generated biset functor. Let $G\in \mathcal{D}$. 
\begin{enumerate}
\item $Rad\big(F(G)\big)\subseteq [Rad(F)](G)$.
\item If the simple quotients of $F$ do not vanish at $G$, then $Rad\big(F(G)\big) =[Rad(F)](G)$.
\end{enumerate}
Let $P$ be a projective indecomposable biset functor. If the simple quotient of $P$ does not vanish at $G$, then $P(G)$ is a projective indecomposable $\D(G,G)$-module.
\end{lemma}
\begin{proof}
Let $M$ be a maximal subfunctor of $F$. Then, $M(G)$ is a maximal submodule of $F(G)$ if the simple quotient $F/M$ does not vanish at $G$ and $M(G)=F(G)$ otherwise. For the second part, if $N$ is a maximal submodule of $F(G)$, let $\overline{N}$ be the subfunctor of $F$ generated by $N$. There is a maximal subfunctor $M$ of $F$ such that $\overline{N}\subseteq M\subset F$. We have $\overline{N}(G)=N\subseteq M(G)\subset F(G)$. By maximality, $M(G)=N$. The result follows. 
\newline\indent If $P$ is a projective indecomposable functor, it has a simple top $S$. By hypothesis, the simple functor $S$ does not vanish at $G$. By Yoneda's Lemma, there is a non-zero morphism between the representable functor $Y_G$ and $S$. Since $S$ is simple, this morphism is surjective. Moreover, the functor $P$ is a projective cover of $S$, so $P$ is isomorphic to a direct summand of $Y_{G}$. This implies that $P(G)$ is a direct summand of $\D(G,G)$. In particular, the evaluation $P(G)$ is a projective $\D(G,G)$-module. Moreover, the module $P(G)$ has a unique simple quotient $S(G)$, so it is indecomposable. 
\end{proof}
\begin{re}
The results of Section $9$ of \cite{bst1} can by illuminated by this Lemma.
\end{re}
As corollary, we also have the following useful result.
\begin{coro}\label{burnside_proj}
Let $G$ be a finite group. Let $k$ be a field. Then, the Burnside module $kB(G)$ is an indecomposable projective $kB(G,G)$-module. It is a projective cover of the simple $kB(G,G)$-module $S_{1,k}(G)$. 
\end{coro}
\begin{proof}\label{principal}
Since $kB$ is a representable functor in $\mathcal{F}_{G,k}$, it is projective. The functor $kB$ is nothing but the functor $L_{1,k}$. By the arguments detailed below Definition \ref{double_burnside},  this functor has a simple top which is isomorphic to $S_{1,k}$. Since the trivial group $1$ is a quotient of the group $G$, one can write $Id_{G/G} = \Def^{G}_{G/G} \circ \Inf_{G/G}^{G}$. By definition, $S_{1,k}(G/G)\cong k$ and since the identity of $S_{1,k}(G/G)$ factorizes through $S_{1,k}(G)$, the last cannot be zero. So, the simple functor $S_{1,k}$ does not vanish at $G$. The result follows from Lemma \ref{ev}.  
\end{proof}
Let $k$ be a field. As explained above, the simple biset functors are parametrized by the pairs $(H,V)$ where $H$ runs through the isomorphism classes of objects of $\mathcal{D}$ and $V$ runs through the isomorphism classes of simple $k\mathrm{Out}(H)$-modules. Via evaluation at $G$, we have a classification of the simple $k\D(G,G)$-modules and their projective cover. 
\begin{coro}\label{lift_proj}
Let $k$ be a field. Let $\D$ be an admissible biset category. The  set consisting of the $S_{H,V}(G)$ for $(H,V)\in \Lambda$ such that $S_{H,V}(G)\neq 0$ is a complete set of representatives of simple $\D(G,G)$-modules. Moreover, $P_{H,V}(G)$ is a projective cover of $S_{H,V}(G)$. 
\end{coro}
\begin{proof}
The first part is Corollary $3.3$ \cite{bst1} which has to be adapted to the case of admissible biset categories, but one more time, this is straightforward. The second part is an obvious consequence of Lemma \ref{ev}.
\end{proof}
Finally we have another family of indecomposable biset functors.
\begin{coro}
Let $k$ be a field. Let $\D$ be an admissible biset category and $G$ be a finite group. The $\D(G,G)$-modules $\Delta_{H,V}(G)$ are indecomposable for $(H,V) \in \Lambda$ such that $S_{H,V}(G) \neq 0$.
\end{coro}
\begin{proof}
We know that $\Delta_{H,V}$ is a quotient of $P_{H,V}$. Since the evaluation functor is exact, we have that $\Delta_{H,V}(G)$ is a quotient of the indecomposable projective module $P_{H,V}(G)$. The result follows. 
\end{proof}
Obviously, this classification is not completely satisfying, as it is well known that understanding which simple biset functors vanish at $G$ is an extremely hard combinatorial problem (see \cite{bst2} for a recent survey).  
\section{The generating relation}
Since it seems too difficult to understand when a simple functor vanishes at a finite group $G$, we try to avoid the difficulty by considering finite groups such that there is no non-trivial vanishing of simple functors.
\begin{de}\label{nv_group}
Let $k$ be a field. Let $G$ be a finite group. The group $G$ is a non vanishing group over $k$ if none of the simple functors of $\FSk$ vanishes at $G$. 
\end{de} 
\begin{re}
This is clearly equivalent to the fact that $S_{H,V}(G) \neq 0$, for every simple functor $S_{H,V}$ of $\mathcal{F}_{k\mathcal{C},k}$ such that $H$ is isomorphic to a subquotient of $G$. 
\end{re}
We use the notation $G$ is a $\mathrm{NV}_{k}$-group if it is non-vanishing over a field $k$.

Let $H$ and $G$ be two finite groups. The composition in the biset category: 
\begin{align*}
kB(H,G) \times kB(G,H) &\to kB(H,H)\\
 \big(\,_{H}U_G , \,_{G}W_H \big)\ \ \ \  &\mapsto U\times_{G}W. 
\end{align*}
 induces a morphism of $kB(H,H)$-modules. We will abusively denote by $kB(H,G)B(G,H)$ the image of this composition in $kB(H,H)$. It is the submodule of $kB(H,H)$ consisting of the linear combinations of elements of the form $W\times_{G} U$ for $W\in RB(H,G)$ and $U\in RB(G,H)$. Using this composition, we have an intrinsic understanding of the non-vanishing at $G$ of all the simple functors $S$ such that $S(H)\neq 0$. 
\begin{prop}\label{evaluation}
Let $k$ be a field. Let $G$ and $H$ be two finite groups of $\D$. The following are equivalent.
\begin{enumerate}
\item $S_{H,V}(G)\neq 0$ for every $kB(H,H)$-simple module $V$.
\item There is an isomorphism of $kB(H,H)$-modules between $kB(H,G)B(G,H)$ and $kB(H,H)$. 
\item There exists $n\in \mathbb{N}^{*}$ and for $1\leqslant i\leqslant n$ there are $U_i \in kB(H,G)$ and $W_i \in kB(G,H)$ such that $$id_{H} = \sum_{i=1}^{n} U_i\times_{G} W_i.$$ 
\end{enumerate}
\end{prop}
\begin{re}
It is important to remark that the family of the simple functors in the first point is \emph{not} the family consisting of the simple functors with \emph{minimal} group $H$. Indeed, here we consider all the simple functors indexed by $(H,V)$ where $V$ is a simple $kB(H,H)$-module and not only a simple $k\Out(H)$-module. 
\end{re}
\begin{proof}
It is clear that $2$ and $3$ are equivalent. We show that $1$ is equivalent to $2$. 

Let us assume $2.$ Since $V\cong S_{H,V}(H)\neq 0$, then the identity of $S_{H,V}(H)$ is non zero. By hypothesis $Id_{H} \in kB(H,G)B(G,H)$. So, this morphism factorizes through $S_{H,V}(G)$ which must be non zero. 

Conversely, let $V$ be a simple $kB(H,H)$-module. Since $S_{H,V}(G)\neq 0$, then $L_{H,V}(G) \neq J_{H,V}(G)$. By the description of $J_{H,V}$ given in $(\ref{jhv})$, this means that there is an element $$\sum_{i} \phi_{i}\otimes v_{i} \in kB(G,H)\otimes_{kB(H,H)}V$$ and an element $\psi \in kB(H,G)$ such that $\sum_i (\psi \phi_{i})\cdot v_{i} \neq 0$. 

So the action of the element $\phi \times_{G} \big(\sum_{i}\phi_i\big)$ is non zero on $V$. Since $V$ is a simple module, we have: $$kB(H,G)B(G,H)\cdot V = V.$$ It holds for every simple module $V$, so $kB(H,G)B(G,H)$ is not contained in any maximal submodule of $kB(H,H)$. It must be equal to $kB(H,H)$. 
\end{proof}
\begin{de}\label{generated}
Let $k$ be a field. Let $G$ and $H$ be two finite groups. We say that $H$ is $k$-\emph{generated} by $G$, if we have $kB(H,G)B(G,H)= kB(H,H)$. In this case we use the notation $H\vdash_k G$. If the context is clear enough, we will simply use the notation $H\vdash G$. 
\end{de}
As immediate Corollary, we have an intrinsic definition of \emph{non-vanishing} groups. 
\begin{prop}\label{eq_def_nv}
Let $G$ be a finite group and $k$ be a field. Then, the following are equivalent.
\begin{itemize}
\item The group $G$ is $\NV_k$.
\item Every subquotient $H$ of $G$ is $k$-generated by $G$. 
\end{itemize}
\end{prop}
\begin{proof}
It is an easy consequence of Lemma \ref{evaluation}. 
\end{proof}
The generating relation also has an algebraic interpretation in the category of biset functors.
\begin{lemma}\label{yoneda_fac}
Let $k$ be a field. Let $\D$ be a full subcategory of the biset category. Let $G$ and $H$ be two groups of $\D$. Then the following are equivalent.
\begin{enumerate}
\item $H \vdash_k G$.
\item There exists $n\in \mathbb{N}^*$ such that $Y_H$ is a direct summand of $(Y_G)^m$. 
\end{enumerate}
\end{lemma}
\begin{proof}
It is an easy application of Yoneda's Lemma. More precisely, if $H \vdash_k G$, then there exists $n\in \mathbb{N}^{*}$ and for $1\leqslant i\leqslant n$ there are $U_i \in kB(H,G)$ and $W_i \in kB(G,H)$ such that $$id_{H} = \sum_{i=1}^{n} U_i\times_{G}W_i.$$ 
By Yoneda's Lemma the biset $U_i$ induces via right multiplication a morphism $\phi_i$ between $Y_H$ and $Y_G$ and the biset $V_i$ induces via right multiplication a morphism $\psi_i : Y_G \to Y_H$. Moreover, the morphism $\sum_{i=1}^{n} \psi_i\circ\phi_i $ corresponds to the identity of $kB(H,H)$ via the isomorphism of Yoneda's Lemma, so it is the identity of $Y_H$. 

Conversely, if there exists $n\in \mathbb{N}$ such that $Y_H$ is a direct summand of $(Y_G)^n$, then for $i \in \{1,\cdots, n\}$ there are morphisms $\phi_i : Y_H \to Y_G$ and $\psi_i : Y_G \to Y_H$ such that:
$$id_{Y_H} = \sum_{i=1}^{n} \psi_i \circ \phi_i.  $$
Using Yoneda's Lemma one more time, we see that $H\vdash_k G$.
\end{proof}
Now, it is not difficult to prove that the \emph{non-vanishing} groups are exactly the finite groups $G$ such that the evaluation at $G$ induces an equivalence of categories between $ \FSk$ and $kB(G,G)\Mod.$ 
\begin{theo}\label{ev_equiv}
Let $G$ be a finite group. Let $k$ be a field. Then the following are equivalent.
\begin{enumerate}
\item $G$ is a $\mathrm{NV}_k$-group. 
\item $ev_{G} : \FSk \to \kbg\Mod$ is an equivalence of categories. 
\end{enumerate}
\end{theo}
\begin{proof}
If the evaluation at $G$ is an equivalence of categories, it cannot kill a simple functor. So $G$ is a non-vanishing group. 

Conversely, Lemma \ref{yoneda_fac} implies that the representable functor $Y_G$ is a pro-generator of $\FSk$. By Morita's Theorem, the functor $\Hom_{\FSk}(RB_{G},-)$ is an equivalence of categories between $\FSk$ and $\End_{\FSk}(RB_{G})$. Finally, Yoneda's Lemma identifies this functor with the evaluation at $G$ and $\End_{\FSk}(RB_{G})$ with  $\kbg$. 

Alternatively, we give a direct proof of the result without using Morita's Theorem. The functor $L_{G,-}$ is a left adjoint to the evaluation at $G$. Let $V$ be a $kB(G,G)$-module. The value at $V$ of the co-unit of the adjunction is the canonical isomorphism $$V\cong kB(G,G)\otimes_{kB(G,G)} V$$ Let $F$ be a k-biset functor over $G$. The value at $F$ of the unit is the following morphism. Let $X$ be a subquotient of $G$, then we have:
\begin{align*}
\epsilon_F(X) : kB(X,G) & \otimes_{kB(G,G)}F(G) \to F(X) \\
 & W  \otimes f \mapsto F(W)(f).
\end{align*}
If $G$ is a $\NV_k$-group, then by Proposition \ref{evaluation} there is an integer $n$ and some morphisms $U_{1},\cdots, U_{n} \in kB(X,G)$, and $W_{1},\cdots W_{n}\in kB(G,X)$ such that $$id_{X} = \sum_{i=1}^{n}U_i \times_{G} W_i.$$ We let $\delta_{F}(X) : F(X) \to kB(X,G)\otimes_{kB(G,G)}F(G)$ be the morphism defined by $\delta_{F}(X)(x):= \sum_{i=1}^{n} U_i \otimes F(W_i)(x)$ for $x \in F(X)$. It is easy to check that $\delta_{F}(X)$ is an inverse isomorphism of $\epsilon_{F}(X)$. 
\end{proof}
\begin{re}\label{quantic}
The arguments developed here are much more general than the case of a replete biset category. 
\begin{enumerate}
\item For example one can take for the category $\D$ a category of so-called left-free, or bi-free, biset functors. At least, over a field of characteristic zero it seems that there are a lot of non-vanishing groups for the left-free case. However, since the category $\D$ contains less morphisms, there are a lot of vanishing of simple functors when the characteristic of the field is non zero, particularly for the bi-free case. 
\item On the other hand, the opposite phenomenon can appear. In the context of correspondence functors recently developed by Bouc and Thévenaz (see \cite{bt_representation_sets} for lots of details ) every object of their category is generated by a larger object. So there is no non-trivial vanishing. As consequence, in this context, if $X$ is a finite set, let $\D$ be the full subcategory of the category of correspondences consisting of all the sets smaller than $X$. Then, the evaluation at $X$ induces an equivalence between the category of $k$-linear functors from $\D$ to $k\Mod$ and the category of modules over the algebra of all relations on $X$. This, together with results on the simple correspondences functors implies, and strengthens Theorem $4.1$ of \cite{bt_relation_alg}.
\item Another example of such equivalence appears in the context of Mackey functors. There are various possible definitions of Mackey functors, for example they can be defined as modules over the Mackey algebras as well as particular bivariant functors over the category of $G$-sets. Unfortunately, the equivalence between the different definitions is rather technical. For a recent survey on theses equivalences see Section $2$ of \cite{equiv_p_loc}, or the first Sections of $\cite{tw_structure}$. With Lindner's definition (see \cite{lindner} ) a Mackey functor is nothing but an additive functor from a category of spans of $G$-sets to the category of abelian groups. Every finite $G$-set  $X$ is generated by the $G$-set $\Omega_{G} := \bigsqcup_{H\leqslant G} G/H$ in the sense of Definition \ref{generated}. So the evaluation at $\Omega_{G}$ induces an equivalence of categories between the category of Mackey functors and the category of modules over $B(\Omega_{G}^2)$ the algebra of endomorphisms of $\Omega_{G}$. This last algebra is known as the Mackey algebra introduced by Th\'evenaz and Webb in \cite{tw_structure}. A similar result holds for cohomological Mackey functors and cohomological Mackey algebras (see Section $2$ of \cite{equiv_coho} for more details about the different definitions of cohomological Mackey functors.)
\end{enumerate}
\end{re}
The following result is now obvious but still interesting.
\begin{coro}
Let $k$ be a field and $G$ be a $\mathrm{NV}_k$-group. Let $F\in \FSk$ be a biset functor over $G$. Then $F\cong L_{G,F(G)}$. 
\end{coro}
\begin{proof}
Since $G$ is a non-vanishing group, the evaluation at $G$ is an equivalence of categories from $\FSk$ to $kB(G,G)\Mod$. Any quasi-inverse equivalence is isomorphic to the left adjoint $L_{G,-}$ of the evaluation. 
\end{proof}
\begin{re}
It is clear that this result fails if $G$ is a vanishing group. Indeed, let $S$ be a simple functor such that $S(G)=0$. The functor $S$ is a non-zero functor, so it cannot be isomorphic to $L_{G,S(G)}=0$. Still, we will use a weak degeneration of this result for the proof of Theorem \ref{semi_simple}.
\end{re}
More generally, if $G$ is a vanishing group, we have a situation of recollement:
\begin{prop}
Let $k$ be a field. Let $\D$ be a replete biset category. Let $G\in \D$. We denote by $\mathcal{K}(G)$ the full subcategory of $\FDk$ consisting of the functors $F$ such that $F(G)=0$. Then, we have a situation of recollement:
\begin{equation*}
\mathcal{K}(G) \underset{\leftarrow}{\overset{\leftarrow}{\to}} \FDk \underset{\leftarrow}{\overset{\leftarrow}{\to}} kB(G,G)\Mod.
\end{equation*}
In particular, $kB(G,G)\Mod \cong \FDk/\mathcal{K}(G)$. 
\end{prop}
\begin{proof}
We give the different functors between these categories. The result will follow from straightforward verifications of the Axioms of \cite{recollement}. The functor between $\FDk$ and $kB(G,G)\Mod$ is the evaluation at $G$. It has a left adjoint $L_{G,-}$ and a right adjoint $L_{G,-}^{o}$ (see $3.3.5$ of \cite{bouc_biset}). The functor between $\mathcal{K}(G)$ and $\FDk$ is the embedding functor. It has a left adjoint which sends a functor $F$ to its largest quotient which belongs in $\mathcal{K}(G)$. The right adjoint is the functor sending $F$ to its largest subfunctor belonging in $\mathcal{K}(G)$. 
\end{proof}
\section{Some non-vanishing groups}
In this section, we investigate basic properties of non-vanishing groups. In the first part we give an infinite list of non-vanishing groups. The groups of this list have the particularity of being non-vanishing over \emph{any} field.  Moreover, they are nilpotent. 

Unfortunately we do not succeed to classify the non-vanishing groups. Our problem is in theory much easier than the problem of understanding the vanishing of the simple functors. Indeed we ask that none of the simple functors vanishes at $G$ which is way stronger. In particular, this is can be rephrased as a condition involving a composition of bisets (see Proposition \ref{eq_def_nv}). 

If this condition seems easier to handle, in the fact it is still difficult to check it. At the end of this section we construct a non-trivial example of non-vanishing group over a field of characteristic zero. We will see that the non-vanishing property of this group comes from some non trivial results about the representations of the group $GL(3,2)$. This example is so pathologic that it seems very unlikely to find a classification. This example also shows that the family of non-vanishing groups contains some \emph{non-nilpotent} groups and is \emph{not} closed under taking subgroups and subquotients and that the non-vanishing condition depends of the ground field.   

Let us start this investigation by looking at the relation $\vdash$. 
\begin{lemma}
Let $k$ be a field.
\begin{itemize}
\item The relation $\vdash_k$ is reflexive, and transitive. 
\item If $G$ and $H$ are two groups such that $H\vdash_k G$ then $H$ is isomorphic to a subquotient of $G$.
\end{itemize}
\end{lemma}
\begin{proof}
Let $G$, $H$ and $K$ be three finite groups. It is clear that $G\vdash_k G$. Now, let us assume that $K\vdash_k H$ and $H\vdash_k G$. One can use the equivalent assertions of Lemma \ref{evaluation} in order to show that $K\vdash_k G$. Alternatively, one can use the equivalent characterisation of Lemma \ref{yoneda_fac}. That is $K\vdash_k H$ if and only if there is an integer $n$ such that $Y_{K} \mid (Y_{H})^{n}$ in the category of biset functors. 
\newline For the second point, if $H\vdash_{k} G$, then $kB(H,G)B(G,H) = kB(H,H)$. Let us denote by $I(H,H)$ the submodule of $kB(H,H)$ consisting of all the $H$-$H$-bisets factorising strictly below $H$. Then $k\Out(H)\cong kB(H,G)B(G,H)/I(H,H)$. In particular, the last quotient is non zero. Moreover, by Bouc's butterfly lemma (See formula $(\ref{butterfly})$) any element in $kB(H,G)$ is a $k$-linear combination of transitive bisets which factors through subquotients of $H$ and $G$. So if $H$ is not a subquotient of $G$, then every $H$-$G$-biset factorizes through a proper subquotient of $H$, and hence is zero in the quotient. 
\end{proof}
\begin{lemma}\label{abelian}
Let $k$ be a field. If $G$ is an abelian group, then $G$ is a $\mathrm{NV}_k$-group. 
\end{lemma}
\begin{proof}
This is Proposition $3.2$ of \cite{bst2}. Since the argument is both crucial and easy we recall it. We have to prove that for every subquotient $H$ of $G$ and for every simple $k\Out(H)$-module, we have $S_{H,V}(G)\neq 0$. Since the group $G$ is abelian, every subquotient is isomorphic to a quotient of $G$. Now if $H=G/N$ for a normal subgroup $N$ of $G$, we have $id_{H} = \Def^{G}_{G/N}\Inf_{G/N}^{G}$.
\end{proof}
More generally, this argument can be generalized to self-dual groups. 
\begin{coro}
Let $G$ be a finite group such that every subgroup is isomorphic to a quotient of $G$. Then $G$ is a $\NV_k$-group for every field $k$. 
\end{coro}
\begin{proof}
By hypothesis, every subgroup of $G$ is isomorphic to a quotient of $G$, so it is generated by $G$. Let's assume that for every $H$ subgroup of $G$ we have $H \vdash_k G$. Let $H=B/A$ be a subquotient of $G$. Then $H$ is generated by $B$, and by hypothesis $B$ is generated by $G$. So by transitivity of the relation $\vdash_k$, the group $H$ is generated by $G$. 
\end{proof}
Such a finite group is called a \emph{$s$-self dual} group and these groups have been completely classified in \cite{s-selfdual}. This classification involves two families of finite $p$-groups. Let $p$ be a prime number, then we denote by $X_{p^3}$ an extra-special $p$-group of exponent $p$. Let $n$ be an integer, then we denote by $M_{p}(n,n)$ the finite group ${<}a,b\ \mid\ a^{p^n}=b^{p^n}=1,\ bab^{-1} = a^{1+p^{n-1}}{>}$. The following theorem summarizes the classification in \cite{s-selfdual}.
\begin{theo}
Let $G$ be a finite group. 
\begin{enumerate}
\item The group $G$ is $s$-self dual if and only if $G$ is nilpotent and all Sylow subgroups of $G$ are $s$-self dual.
\item Let $p$ be a prime number. Let $P$ be a finite $p$-group. Then $P$ is $s$-self dual if and only if $P$ is:
\begin{itemize}
\item $P$ is abelian.
\item $P \cong X_{p^{3}}\times M$ where $M$ is an abelian $p$-group with $\mathrm{exp}(M)\leqslant p$ when $p$ is an odd prime number. 
\item $P\cong M_p(n,n)\times M$ where $M$ is an abelian $p$-group with $\mathrm{exp}(M)<p^{n}$.
\end{itemize}
\end{enumerate}
\end{theo}
\begin{proof}
Theorem $7.3$ together with Theorem $7.1$ of \cite{s-selfdual}. 
\end{proof}
Before continuing our investigation let us recapitulate what we have done: the list of non-vanishing groups includes all the \emph{abelian groups} and this list of self-dual groups.  

The proof follows from some easy considerations about the generating relations and the trivial fact that if $H$ is isomorphic to a quotient of $G$, then $S_{H,V}(G)\neq 0$ for any simple $k\Out(H)$-modules. 

As it can be seen in Section $3$ of $\cite{bst2}$ there are some less trivial non-vanishing properties involving the geometry of the sections in a finite group. So, we cannot hope to have constructed all the non-vanishing groups at this stage!

There is a formula for the dimension of the evaluation at a finite group $G$ of the simple functor $S_{H,V}$ in Theorem $7.1$ of \cite{bst1}. This formula involves the computation of a bilinear form which is constructed by using the \emph{character} of the simple $k\Out(H)$-module $V$. In particular, the vanishing or the non-vanishing at $G$ of this simple functor \emph{should depend} on the simple module $V$ and on the field $k$. 

However, we are interested in the non vanishing at $G$ of \emph{all} the simple functors having $H$ as minimal group. We have several results for this type of global non-vanishing (See Section $3$ of \cite{bst2} for more details) and all of them only depend on the finite group (not on the ground field nor on the simple modules). In these cases, the generating relation is very simple because the identity is the product of two transitive bisets. 

So, it seems natural to ask the following two questions:
\begin{question}
Let $k$ be a field. Let $G$ be a $\NV_k$-group. Let $H$ be a subquotient of $G$. Is it always possible to find an element $U \in kB(H,G)$ and an element $V\in kB(G,H)$ such that $id_H = U \times_G V$ ?
\end{question}
This question is not hopeless since the corresponding question has a positive answer for the category of finite sets with correspondences (See Lemma $4.1$ of \cite{bt_representation_sets}). 
\begin{question}
Let $k$ be a field and $G$ be a finite group. If $G$ is non-vanishing over the field $k$, is it non-vanishing over any field? 
\end{question}
In the rest of this section, we are going to give a negative answer to these two naive questions. More precisely, we give an example of a non-vanishing group $G$ over a field of characteristic zero with a subquotient $H$ such that $id_{H}$ is \emph{not} a product of two elements. Moreover, this example shows that the structure of the simple modules for the group algebra of the outer automorphism of the subquotients of $G$ is also involved in the vanishing or non-vanishing property of $G$. In particular, this group is not non-vanishing over any field. We start by a useful general lemma about non-vanishing groups.
\begin{lemma}\label{extension}
Let $G$ be a finite group. Let $k\subset K$ be a field extension. Then $G$ is $\NV_k$ if and only if $G$ is $\NV_K$. 
\end{lemma}
\begin{proof}
By Lemma \ref{evaluation}, the group $G$ is $NV_k$ if and only if $kB(H,G)B(G,H) = kB(H,H)$ for every $H \sqsubseteq G$. That is, if and only if the $kB(H,H)$-modules $kB(H,G)B(G,H)$ and $kB(H,H)$ are isomorphic for every $H\sqsubseteq G$. Since $K\otimes_k kB(H,H) = KB(H,H)$ and $K\otimes_k kB(H,G)B(G,H)=KB(H,G)B(G,H)$, the result follows from Noether Deuring's Theorem (See Theorem $19.25$ \cite{lam}). 
\end{proof} 
We will also use the following result in order to compute the dimension of some evaluations of some simple functors. 
\begin{theo}[Bouc]\label{dim}
Let $k$ be a field of characteristic $0$. Let $p$ be a prime number and $P$ be a finite $p$-group which is not $1$ or $C_p\times C_p$. Let $G$ be a finite group. Then, $\dim_k S_{P,k}(G)$ is the number of conjugacy classes of sections $(T,S)$ of $G$ such that $T/S \cong P$ and $T$ is the direct product of a p-group and a cyclic group.  
\end{theo}
\begin{proof}
See the main Theorem of \cite{bouc_simple}.  
\end{proof}
\begin{lemma}
Let $k$ be a field of characteristic $0$. Let $G=A_4 \times C_2$ and $H=C_2 \times C_2 \times C_2$. Then $id_H$ cannot be written as a product of an element of $kB(H,G)$ and an element of $kB(G,H)$.
\end{lemma}
\begin{proof}
Let us first remark that the fact that $id_H$ is the product of an element of $kB(H,G)$ and an element of $kB(G,H)$ is equivalent to the fact that $Y_H$ is a direct summand of $Y_G$ in the category of biset functors. Moreover, by standard arguments, if $X$ is a finite group, then the decomposition of the Yoneda functor $Y_X$ as direct sum of indecomposable projective is given by:
\begin{equation*}
Y_X \cong \bigoplus_{(K,W) \in \Lambda} P_{K,W}^{n_{K,W}(X)},
\end{equation*}
where $P_{K,W}$ is a projective cover of $S_{K,W}$ and 
$$n_{K,W}(X)= \dim_k S_{K,W}(X)/\dim_k \End(W). $$ 
So $Y_H$ is a direct summand of $Y_G$ if and only if for every simple functor $S_{K,W}$ we have:
\[\dim_k S_{K,W}(H) \leqslant \dim_k S_{K,W}(G). \] 
Now, we claim that we have $\dim_k S_{C_2,k}(G)=14$ and $\dim_k S_{C_2,k}(H)=35$ so $Y_H$ is not a direct summand of $Y_G$. It remains to prove the claim. These computations can be, in theory, done by using the arguments of Paragraph $7$ of \cite{bst1}. However, the bilinear forms are so huge that it is not reasonable to give a full proof here. Instead, we use Theorem \ref{dim}. Since there are $35$ sections $C_2$ in $C_2\times C_2\times C_2$, we have that $\dim_k S_{C_2,k}(H)=35$. It is not difficult to check that there are $15$ conjugacy classes of sections $C_2$ in $A_4 \times C_2$: $3$ sections $(C_2,1)$, $1$ section $(C_6,C_3)$, $7$ sections $(C_2\times C_2,C_2)$, $3$ sections $\big((C_2)^3,(C_2)^2\big)$ and $1$ section $(A_4\times C_2, A_4)$. However, since $A_4\times C_2$ is not the direct product of a $2$-group and a cyclic group, we have to discount the conjugacy class of the section $(A_4\times C_2,A_4)$. Finally, we have $\dim_k S_{C_2,k}(G)=14$.
\end{proof}
The next step is to check that $A_4\times C_2$ is a non-vanishing group over a field of characteristic zero. Actually, we prove that it is non-vanishing over any field of characteristic different from $3$. 
\begin{lemma}
Let $G=A_4\times C_2$ and let $k$ be a field. The group $G$ is a $\NV_k$-group if and only if $\mathrm{char}(k)\neq 3$. 
\end{lemma}
\begin{proof}
Using the fact that the relation $\vdash\ :=\ \vdash_k$ is transitive, it is enough to check that $H\vdash G$ for every subgroup $H$ of $G$. Up to isomorphism, the subgroups of $G$ are: $1$, $C_2$, $C_3$, $C_2\times C_2$, $C_6$, $C_2\times C_2\times C_2$, $A_4$ and $A_4\times C_2$. For our purpose it is enough to look at subgroups which are not isomorphic to a quotient of $G$. It remains $C_2\times C_2$ and $H:=C_2\times C_2\times C_2$. Moreover, the group $C_2\times C_2$ is a quotient of $H$, so by transitivity of $\vdash$, it is enough to check that $H\vdash G$. So we want to check that
 \[id_{H} \in kB(H,G)kB(G,H),\] 
 but in this case the identity is not a product of bisets, and it is rather technical to check this without the help of a computer. However, this is equivalent to the checking that for every simple $k\Out(H)$-module $V$ we have $S_{H,V}(G)\neq 0$. Since $\Sigma_{H}(G) = \{ (H,1) \}$, by Proposition $7.1$ of \cite{bst2}, we have: 
\begin{equation*}
S_{H,V}(G) \cong \mathrm{Tr}_{1}^{G/H}(V),
\end{equation*}
where $G/H$ acts on $V$ via conjugation. 

Since $G/H$ is a cyclic group of order $3$, when $k$ is a field of characteristic $3$, then $S_{H,k}(G)=0$ and $G$ is a vanishing group. Now let us suppose that $k$ is a field of characteristic $p\neq 3$. By Lemma \ref{extension}, we can assume the field to be algebraically closed. Let $\Gamma$ be the subgroup of $GL_{3}(\mathbb{F}_2)$ consisting of $\mathrm{conj}_{g}$ for $g\in G/H$. Let $V$ be a simple $kGL_{3}(\mathbb{F}_2)$-module. Since $\Gamma$ is a group of order $3$, the group algebra $k\Gamma$ is semisimple. So $\Res^{GL_{3}(\mathbb{F}_2)}_{\Gamma}(V)$ is a semisimple module. The space of fixed point of either of the two non trivial simple modules is $0$. In conclusion $V^{\Gamma} \neq \{0\}$ if and only if the trivial $k\Gamma$-module is a direct summand of $\Res^{GL_{3}(\mathbb{F}_2)}_{\Gamma}V$. Since the trace map is surjective, $S_{H,V}(G)\neq 0$ if and only if $k\mid \Res^{GL_{3}(\mathbb{F}_2)}_{\Gamma} V$. We did not find a structural reason for this fact, and as far as we known, it is maybe just a coincidence for the group $GL_3(\mathbb{F}_2)$ and its subgroup of order $3$. 

There are basically three cases: $p=0$, $p=2$ and $p=7$. This can be done by looking at the ordinary character table for $p=0$ and for $p=2$ and $p=7$, one can look at the Brauer character tables of $GL_{3}(\mathbb{F}_2)$. For our purpose, we only need to have the values of the characters on the elements of order $1$ and $3$. All the elements of order $3$ are conjugate in $G$. We let $x$ to be an element of order $3$. We have the following tables, where the characters are written horizontally. 
\begin{equation*}
\begin{tabular}{|c|c|c|c|c|c|c|}
\hline 
(1) & 1 & 3 & 3 & 6 & 7 & 8 \\ 
\hline 
(x) & 1 & 0 & 0 & 0 & 1 & -1  \\ 
\hline 
\end{tabular} 
\end{equation*}
For the prime $p=2$ we have the following table:
\begin{equation*}
\begin{tabular}{|c|c|c|c|c|}
\hline 
(1) & 1 & 3  & 3 & 8 \\ 
\hline 
(x) & 1 & 0 & 0 & -1 \\ 
\hline 
\end{tabular} 
\end{equation*}
And finally, for $p=7$, we have:
\begin{equation*}
\begin{tabular}{|c|c|c|c|c|}
\hline 
(1)& 1 & 3 & 5 & 7 \\ 
\hline 
(x) & 1 & 0 & -1 & 1 \\ 
\hline 
\end{tabular} 
\end{equation*}
In the three cases, it is easy to check that the trivial module appears at least ones in $\Res^{GL(3,2)}_{C_3}(V)$ for every simple $kGL(3,2)$-module $V$. 
\end{proof}
As corollary, of this example, we have:
\begin{coro}\label{ex_vanishing}
Let $k$ be a field. The family of $\NV_k$-groups is not closed under taking subgroups or taking quotients. Moreover, it contains non nilpotent groups. 
\end{coro}
\begin{proof} $A_4$ is a subgroup and a quotient of $A_4\times C_2$. We just have to show that $A_4$ is a vanishing group for every field. Let $C_2\times C_2$ be the subgroup of order $4$ of $A_4$. It is easy to see that $\Sigma_{C_2\times C_2}(A_4) = \{ (C_2\times C_2,1)\}$. So if $V$ is a simple $k\Out(C_2\times C_2)$-module, by Proposition $7.1$ of \cite{bst2}, we have: 
\begin{equation*}
S_{H,V}(G) \cong \mathrm{Tr}_{1}^{N_{G}(T,S)}(V).
\end{equation*} 
Here $\Out(C_2\times C_2)\cong S_3$ and $N_{G}(T,S)$ is the subgroup of $S_3$ of order $3$.
\begin{enumerate}
\item If $\mathrm{char}(k)=0$ or $2$. If $V$ is the simple $kS_3$-module of dimension $2$, then it is easy to check that $Tr_{1}^{C_3}(V)=0$.
\item If $\mathrm{char}(k)=3$, then $Tr_{1}^{C_3}(k)=0$.
\end{enumerate}
So in every case, the group $G$ is a vanishing group. 
\end{proof}
Finally, we summaries the known results for the non-vanishing groups.
\begin{theo}
\begin{enumerate}
\item The abelian groups and, more generally, the $s$-self dual groups are non-vanishing over any field.
\item There are other examples of non-vanishing groups over a field of characteristic $0$.  
\item The family of non-vanishing group over a field $k$ is not closed under taking subgroups or quotients. 
\end{enumerate}
\end{theo}
\begin{re}
I don't know whether there are other examples of groups that are non-vanishing over any field. Even better, one can define the generating relation over the ring of integers. Then all the $s$-self-dual groups are non-vanishing over the ring $\mathbb{Z}$, and I am wondering about the existence of others. 
\end{re}
\section{Around highest-weight structure }\label{qh}
We recall the \emph{famous Theorem} of Peter Webb about the Highest-weight structure of the category of biset functors.
\begin{theo}[Webb]\label{webb}
Let $\D$ be an admissible category. Let $k$ be a field such that $\mathrm{char}(k) \nmid |\Out(H)|$ for every $H\in \D$. If $\D$ has only finitely many isomorphism classes of objects, then the category $\FDk$ is a highest-weight category.
\end{theo}
\begin{proof}
This is a reformulation of Theorem $7.2$ of \cite{webb_strat2}. The standard objects are given by the functors $\Delta_{H,V}$ where $S_{H,V}$ is a simple functor of $\D$. Our context is slightly more general than the one of Webb. Indeed, his theorem is stated in terms of \emph{globally defined Mackey functors}. They correspond to some particular admissible biset categories. Nevertheless, it is straightforward to check that his result can be extended to our more general situation. Moreover, one can avoid the counting arguments of Theorem $6.3$ of \cite{webb_strat2} by a systematic use of this functorial definition of the standard objects. 
\end{proof}
The representation theory of quasi-hereditary algebras is philosophically close to the representation theory of semi-simple Lie algebras. One very important feature of semi-simple Lie algebra is the existence of so-called Borel subalgebras. This notion has been generalized to arbitrary quasi-hereditary algebras by K{\"o}nig in \cite{konig_borel} and seems to be very important in recent development of the theory. The key result, which is highly non trivial, is that for every quasi-hereditary algebra $A$, there is an algebra $A'$ that is Morita equivalent to $A$ such that $A'$ has an exact Borel subalgebra (See Corollary $1.3$ of \cite{kko_borel}). In the rest of the section, we describe the category of modules over an exact Borel subalgebra of the biset functor category.
\begin{de}[Definition 2.2 of \cite{kko_borel}]
Let $(A,\leqslant)$ be a quasi-hereditary algebra with $n$ simple modules. Then, a subalgebra $B\subseteq A$ is called an exact Borel subalgebra if:
\begin{enumerate}
\item The algebra $B$ has also $n$ simple modules denoted by $L_{B}(i)$ for $i\in \{1,\cdots,n\}$, and $(B,\leqslant)$ is a quasi-hereditary algebra with \emph{simple} standard modules.
\item The induction functor $A\otimes_{B}-$ is exact.
\item There is an isomorphism $A\otimes_{B} L_{B}(i) \cong \Delta_{A}(i)$. Here $\Delta_{A}(i)$ denotes the standard module with weight $i$ of $A$. 
\end{enumerate}
\end{de}

The reason to call such a subalgebra an exact Borel subalgebra comes from an analogy due to K{\"o}nig (See Introduction of \cite{konig_borel}). Let $\mathfrak{g}$ be a complex semi-simple Lie algebra and $\mathfrak{b}$ be a Borel subalgebra. Let $\mathcal{U}(\mathfrak{g})$ and $\mathcal{U}(\mathfrak{b})$ be the corresponding enveloping algebras. The Poincar\'e-Birkhoff-Witt theorem implies that $\mathcal{U}(\mathfrak{g})$ is free as left or right $\mathcal{U}(\mathfrak{b})$-module (See Sections $0.5$ and $1.3$ \cite{Humphreys}). In particular, the induction from $\mathcal{U}(\mathfrak{b})$ to $\mathcal{U}(\mathfrak{g})$ is an exact functor. So the second point of the definition can be thought of as an analogous of the PBW theorem. 

 The category of biset functors is equivalent to a category of modules over an algebra $A$ (without unit in general). But, in this paper we are more interested by the category of biset functors and not very interested by the choice of an underlying algebra. So, if there is a subcategory of $\FDk$ which is equivalent to the category of modules over an exact Borel subalgebra of $A$, I will \emph{abusively} say that this category is an \emph{exact Borel subcategory.}

 If $\D$ is a replete biset category, we denote by $\D_{0}$ the following admissible biset category. The objects of $\D_{0}$ are the objects of $\D$. Now, if $H$ and $K$ are two groups of $\D$, then $\Hom_{\D_0}(H,K)$ is given by the left-free double Burnside algebra. That is, we have forgotten the inflation in the five elementary bisets. The biset functors over $\D_0$ are sometimes called \emph{deflation functors.} 
\begin{lemma}\label{iso1}
Let $\D$ be a replete biset category and let $\D_0$ as above. Let $H\in \D$ and $F\in \FDk$. We write $\Res^{\D}_{\D_0}$ the forgetful functor from $\FDk$ to $\mathcal{F}_{\D_0,k}$. Then, we have an isomorphism, natural in $F$, of $k\Out(H)$-module
$$ \underline{\Res^{\D}_{\D_0}F}(H) \cong \underline{F}(H).$$  
\end{lemma}
\begin{proof}
Let $F\in \FDk$. Then, by definition we have:
\begin{equation*}
\underline{F}(H) = \bigcap_{\underset{U\in kB(K,H)}{K\sqsubset H}} \mathrm{Ker}\big(F(U)\big).
\end{equation*}
We claim that we have:
\begin{equation}\label{3}
\underline{F}(H) = \bigcap_{\underset{B/A \sqsubset H }{(B,A)\in \Sigma(H)}} \mathrm{Ker}\Big(F(\Def^{B}_{B/A} \circ \Res^{H}_{B})\Big). 
\end{equation}
It is clear that $\underline{F}(H)$ is a subset of the right hand side term. Conversely, let $x$ be an element of the right hand side. Let $K$ be a strict subquotient of $H$ and $U\in kB(K,H)$. Then $U$ is a linear combination of transitive bisets $(K\times H)/L$. Moreover, by Bouc's Butterfly decomposition we have:
$$(K\times H)/L \cong \Ind_{D}^{K} \circ \Inf_{D/C}^{D} \circ \Iso(\alpha) \circ \Def^{B}_{B/A} \circ \Res^{H}_{B} ,$$
where $(B,A)$ is a particular section of $H$ and $(D,C)$ is a particular section of $K$. It is clear that $B/A$ is a \emph{strict} subquotient of $H$ as it is isomorphic to a subquotient of $K$.  Now $\Def^{B}_{B/A} \circ \Res^{H}_{B}$ kills $x$ by assumption, so $(K\times H)/L$ kills $x$. This proves that $x\in \mathrm{Ker}(F(U))$, and the equality $(\ref{3})$ holds. The same result holds for $\Res^{\D}_{\D_0}(F)$, as the butterfly decomposition still holds in the category $\D_0$. The only difference is that there is no inflation in this decomposition. 
\end{proof}
\begin{theo}\label{6.4}
Let $\D$ be a replete biset category with only finitely many isomorphism classes of objects. Let $k$ be a field such that $\mathrm{char}(k)\nmid |\Out(H)|$ for $H\in \D$. Then, the category $\mathcal{F}_{\D_0,k}$ is an exact Borel subcategory of $\FDk$. 
\end{theo}
\begin{proof}
By Webb's Theorem \ref{webb}, under these hypothesis the category $\mathcal{F}_{\D_0,k}$ is a highest-weight category. Moreover, by Proposition $9.1$ of \cite{webb_strat2}, all the standard functors $\Delta_{H,V}^{\D_0}$ are simple. Let $V$ be an arbitrary $k\Out(H)$-module. Let $F$ be any functor in $\FDk$. Then, using successive adjunctions and the isomorphism of Lemma \ref{iso1}, we have:
\begin{align*}
\Hom_{\FDk}\Big(\, ^{l}\Ind_{\D_0}^{\D} \big(\Delta_{H,V}^{\D_0}\big), F\Big) &\cong \Hom_{\mathcal{F}_{\D_0,k}}\Big(\Delta_{H,V}^{\D_0},\Res_{\D_0}^{\D}F\Big),\\
&\cong \Hom_{k\Out(H)}\Big(V,\underline{\Res_{\D_0}^{\D}(F)}(H)\Big),\\
&\cong \Hom_{k\Out(H)}\big(V,\underline{F}(H)\big)\\
&\cong \Hom_{\FDk}\big(\Delta_{H,V}^{\D},F\big).
\end{align*}
Since this holds for any functor $F$, we have that $\, ^{l}\Ind_{\D_0}^{\D}$ sends the standard functors for the category $\D_0$ to the standard functors in $\FDk$.

 It remains to check that the functor $\, ^{l}\Ind_{\D_0}^{\D}$ is exact. For this, we use the description of this induction due to Rosalie Chevalley (See Section $5.1$ of \cite{chevalley_these} for more details). Let $F$ be a deflation functor, then she proved\footnote{under stronger hypothesis on the field and the category $\D$. However, one can check that her proof can be generalized to our situation. } that 
\[\,^l (\Ind_{\D_0}^{\D}F\big)(G) = \Big(\bigoplus_{H\leqslant G} F(N_{G}(H)/H) \Big)_G \cong \bigoplus_{[H\leqslant G]} F(N_{G}(H)/H)_{N_{G}(H)}. \]
Here $[H\leqslant G]$ denotes a set of representatives of conjugacy classes of subgroups of $G$, and the group $N_{G}(H)$ acts on $F(N_{G}(H)/H)$ via conjugation. However, the action of $N_{G}(H)$ over $F(N_{G}(H)/H)$ is trivial. So, we have: 
\[\,^l (\Ind_{\D_0}^{\D}F\big)(G) \cong \bigoplus_{[H\leqslant G]} F(N_{G}(H)/H). \]
In particular, it becomes clear that the functor $\, ^l \Ind_{\D_0}^{\D}$ is exact. 
\end{proof}
For the double Burnside algebra, the situation is \emph{more complicated}.
\begin{theo}
Let $G$ be a finite group. Let $k$ be a field such that $\mathrm{char}(k)\nmid |\Out(H)|$ for $H\sqsubseteq G$. 
\begin{enumerate}
\item The left-free double Burnside algebra $kB_{0}(G,G)$ is a quasi-hereditary algebra with simple standard modules.
\item If $G$ is a non-vanishing group, then the double Burnside algebra is quasi-hereditary. However, the left-free double Burnside algebra $kB_{0}(G,G)$ is not always an exact Borel subalgebra. 
\end{enumerate}
\end{theo}
\begin{proof}
\begin{enumerate}
\item We denote by $\Sigma(G)_0$ the subcategory of the biset category consisting of the subquotients of $G$ and where the morphisms are given by the left-free double Burnside modules. In general the category of modules over $kB_{0}(G,G)$ is not equivalent to the category $\mathcal{F}_{\Sigma(G)_0,k}$. Still, by Webb's Theorem \ref{webb}, the category $\mathcal{F}_{\Sigma(G)_0,k}$ is a highest-weight category with simple standard functors. Let $\Lambda$ be the set of $S_{H,V}(G)$ where $S_{H,V}$ runs a set of representatives of simple functors of $\mathcal{F}_{\Sigma(G)_0,k}$ such that $S_{H,V}(G)\neq 0$. Then, by the arguments of Lemma \ref{ev}, the set $\Lambda$ is a complete set of representatives of simple modules over the left-free double Burnside algebra. Moreover, if $P_{H,V}$ is a projective cover of $S_{H,V}$, then $P_{H,V}(G)$ is a projective cover of $S_{H,V}(G)$. The standard modules will be the evaluation at $G$ of the standard functors $\Delta_{H,V}$ such that $S_{H,V}(G)\neq 0$. By hypothesis, there are subfunctors $0=M_0\subset M_1 \subset M_2 \subset \cdots M_n = P_{H,V}$ such that $M_{i}/M_{i-1} \cong \Delta_{H_i,V_i}$ where $V_i$ is a simple $k\Out(H_i)$-module and $H_i$ is a strict subquotient of $H$. Since the evaluation functor is exact, we have a filtration of $P_{H,V}(G)$. Moreover, $M_{i}(G)/M_{i-1}(G) \cong \Delta_{H_i,V_i}(G) = S_{H_i,V_i}(G)$. So this evaluation is either zero or equal to the simple module $S_{H_i,V_i}(G)$. This shows that the sequence of the $M_{i}(G)$ is not a sequence of strict submodules. We let $N_0 \subset N_1 \subset \cdots N_k = P_{H,V}(G)$ be the strict filtration obtained by removing the multiple terms.  Then by construction $N_{i}/N_{i-1} \cong \Delta_{H_j,V_j}(G)$ for some simple module $S_{H_j,V_j}(G)$ such that $H_j$ is a strict subquotient of $H$. This implies that every projective indecomposable module is filtered by standard modules. Moreover the standard modules that appear in a filtration of the projective indecomposable module $P_{H,V}(G)$ are indexed by groups $K$ such that $K$ is isomorphic to a strict subquotient of $H$. Finally, the standard modules are simple.
\item See Theorem $2.1$ of \cite{qh_doubleburnside} for a proof without using the equivalence of Theorem \ref{ev_equiv}. In this case the evaluation at $G$ is an equivalence of categories between $\FSk$ and $kB(G,G)\Mod$. The result follows from the fact that under these hypothesis, the category $\FSk$ is a highest-weight category. If $G$ is a non-vanishing group, this result does not implies that the left-free double Burnside algebra is an exact Borel subalgebra of $kB(G,G)$. Indeed the last one may have more simple objects. For example if $G$ is a $s$-self dual group which is not self dual. That is a group such that every subgroup is isomorphic to a quotient but there is a quotient $H$ which is not isomorphic to a subgroup. This is clearly a non-vanishing group. The left-free bisets of $B_{0}(G,H)$ are linear combinations of transitive bisets of the form
\[  \Ind_{C} ^{G} \circ \Iso(f) \circ \Def^{B}_{B/A} \circ \Res^{H}_{B},\]
where $(B,A)$ is a section of $H$, $C$ is a subgroup of $G$ and $f$ is an isomorphism from $B/A$ to $C$. Since $H$ is not isomorphic to a subgroup of $G$, then $B/A$ has to be a strict subquotient of $H$. In particular the quotient of $kB_{0}(G,H)$ by the ideal consisting of the morphisms factorizing strictly below $H$ is zero. This implies that for any simple $k\Out(H)$-module $V$, we have $\Delta_{H,V}^{\Sigma(G)_0}(G) = 0$. But by Theorem \ref{6.4}, the simple deflation functor $S_{H,V}^{\Sigma(G)_0}$ is isomorphic to the corresponding standard functor. In particular, we have $S_{H,V}^{\Sigma(G)_0}(G)=0$. 
\end{enumerate}
\end{proof}
As corollaries we have:
\begin{coro}
Let $G$ be a finite group. Let $k$ be a field such that $|\Out(H)|$ is invertible in $k$ for every subquotient $H$ of $G$. Then, the global dimension of the left-free double Burnside algebra $kB_0(G,G)$ is finite.
\end{coro}
For the double Burnside algebra, we need more hypothesis. 
\begin{coro}
Let $G$ be a finite group. Let $k$ be a field such that $|\Out(H)|$ is invertible in $k$ for every subquotient $H$ of $G$. Then, if $G$ is a $\NV_k$-group we have:
\begin{enumerate}
\item The global dimension of $kB(G,G)$ is finite.
\item The Cartan matrix of $kB(G,G)$ has determinant $1$. 
\end{enumerate}
\end{coro}
The fact that the double Burnside algebra of a non-vanishing group is quasi-hereditary is a direct consequence of the fact that the category of biset functors is a highest-weight category. However, if the group $G$ is a vanishing group, one can wonder if the double Burnside algebra is still quasi-hereditary. In particular, it may be possible to find a more suitable order on the set of simple modules over the double Burnside algebra in order to avoid the vanishing problems. In \cite{qh_doubleburnside}, we proved that the global dimension of $\mathbb{C}B(A_5,A_5)$ is infinite. In particular, this shows that such a better ordering does not exist for the double Burnside algebra for $A_5$. We have. 
\begin{prop}
The double Burnside algebra $\mathbb{C}B(A_5,A_5)$ is \emph{not} quasi-hereditary. 
\end{prop}
\begin{proof}
Proposition $3.3$ of  \cite{qh_doubleburnside}.
\end{proof}
\section{Semi-simplicity revisited}
In this section we revisit the semi-simple property of the double Burnside algebra and the category of biset functors. By results of Barker and Bouc, we know precisely when both objects are semi-simple. It is clear that the result of Barker on the category of biset functors implies the result of Bouc about the double Burnside algebra. Curiously, Barker's result can be reformulated as follows. Let $\D$ be a replete biset category and $k$ be a field. Then, the category $\FDk$ is semi-simple if and only if the endomorphism algebra of every object of $\D$ is semi-simple. In general, it is easy to find a category where all the endomorphism algebras of objects are semi-simple but its category of representations is not semi-simple. Here we show that this phenomenon is related to the \emph{generating} relation. As corollary, this gives a rather simple proof of Barker's Theorem. We also give a useful characterization of the semi-simple property in terms of the so-called trivial object. More precisely, we show that similarly to the case of group algebras, these categories are semi-simple if and only if the trivial object is projective. This characterization will be used in the last Section of this article. 
\newline\indent We start this section by recalling when the double Burnside algebra is semi-simple. 
\begin{theo}[Bouc]\label{bouc_ss}
Let $k$ be a field and $G$ be a finite group. The double Burnside algebra is semi-simple if and only if $G$ is cyclic and $\mathrm{char}(k)\nmid \phi(|G|)$.
\end{theo} 
\begin{proof}
Proposition $6.1.7$ of \cite{bouc_biset}.
\end{proof}
We need the following Lemma.
\begin{lemma}\label{L_fun}
Let $k$ be a field. Let $\D$ be a replete biset category. Let $H$ and $K$ be two groups such that $H$ is $k$-generated by $K$. Let $V$ be a $kB(H,H)$-module. Then, 
$$L_{K,L_{H,V}(K)} \cong L_{H,V}. $$ 
\end{lemma}
\begin{proof}
The co-unit of the adjunction between the evaluation at $K$ and $L_{K,-}$ gives a morphism from $L_{K,L_{H,V}(K)}$ to $L_{H,V}$. This morphism $\phi$ is defined on a group $G$ by $$\phi_{G}\Big(W\otimes \big(U\otimes v\big)\Big) = \big( W\times_{K} U\big) \otimes v,$$
for $W\in kB(G,K)$, $U\in kB(K,H)$ and $v\in V$. Since $H$ is generated by $K$, there are $U_i \in kB(H,K)$ and $W_i \in kB(K,H)$ for $i=1,\cdots, n$ such that $id_{H} = \sum_{i=1}^{n} U_i \times_{K} W_i$. We define $\psi_{G} : kB(G,H)\otimes_{k} V \to L_{K,L_{H,V}(K)}$ by 
$$\psi_{G}(U\otimes v) = \sum_{i=1}^{n} \big(U\times_{H} U_i\big)\otimes \big( W_i\otimes v\big), $$
where $U\in kB(G,H)$ and $v\in V$. If $\alpha \in kB(H,H)$, we have:
\begin{align*}
\psi_{G}\big((U\times_H \alpha )\otimes v\big) &= \sum_{i=1}^{n} \big(U\times_{H} \alpha \times_{H} U_i \big) \otimes \big( W_i\otimes v\big),\\
&= \sum_{i=1}^{n} \Big( U \times_{H} (\sum_{j=1}^{n}U_j\times_K W_j)\times_H \alpha \times_{H} U_i \Big) \otimes \big( W_i\otimes v\big),\\
&= \sum_{i,j = 1}^{n} \big( U\times_{H} U_j \big)\otimes \Big(\big(W_j \times_H \alpha \times_H U_i \times_K W_i\big)\otimes v\Big),\\
&=\sum_{j = 1}^{n} \big( U\times_{H} U_j \big)\otimes \Big(W_j \times_H \alpha \times_H\big(\sum_{i=1}^{n} U_i \times_K W_i\big)\otimes v\Big),\\
&=\psi_{G}\big(U \otimes(\alpha \times_H v)\big).
\end{align*}
So $\psi$ can be factorized as a morphism from $kB(G,H)\otimes_{kB(H,H)} V$ to $L_{K,L_{H,V}(K)}$. Moreover, it is clear that $\phi_G$ and $\psi_G$ are two inverse isomorphisms. 
\end{proof}
For the category of biset functors, we have the following Theorem.  Barker's theorem (Theorem 1 of \cite{barker_biset}) is exactly the equivalence of $1.$ and $3.$ with slightly stronger hypothesis on the characteristic of the field $k$. 
\begin{theo}[Barker]\label{semi_simple}
Let $k$ be a field. Let $\D$ be a replete biset category. Then, the following are equivalent.
\begin{enumerate}
\item The category $\FDk$ is semi-simple.
\item For every group $H\in \D$, the algebra $kB(H,H)$ is a semi-simple algebra.
\item Every group $H$ of $\D$ is cyclic and $\mathrm{char}(k) \nmid \phi(|H|)$.
\end{enumerate}
\end{theo}
\begin{proof}
Let $G \in \D$. If $\FDk$ is semi-simple, then the Yoneda functor $Y_G$ is direct sum of simple functors. So its endomorphism algebra is semi-simple. Moreover, by Yoneda's Lemma this last algebra is nothing but the double Burnside algebra of the group $G$. So $1$ implies $2$. And $3$ is nothing but a reformulation of $2$ using Bouc's Theorem \ref{bouc_ss}. 
\newline\indent Now, we will prove that $3$ implies $1$. Let $H \in \D$. Let $V$ be a simple $k\Out(H)$-module. Then by inflation, the module $V$ is also a simple $kB(H,H)$-module. Since the last algebra is semi-simple, the module $V$ is a projective indecomposable $kB(H,H)$-module. Since $L_{H,-}$ is a left adjoint to the exact functor $ev_H$, it sends projective modules to projective functors. Moreover, it sends indecomposable modules to indecomposable functors.  So the functor $L_{H,V}$ is a projective indecomposable functor with simple quotient $S_{H,V}$. In other words, this functor is a projective cover of $S_{H,V}$. It has a unique maximal subfunctor $J_{H,V}$ and this functor has the property of vanishing at $H$. 
\begin{itemize}
\item First let us assume that $\D$ is the full subcategory of the biset category consisting of all the cyclic groups. Let $K$ be a group of $\D$. Let $M$ be a $\mathrm{lcm}$ of the groups $H$ and $K$ (a minimal cyclic group such that $H$ and $K$ are subgroups of $M$). Then, the group $M$ is in the category $\D$. Since $M$ is abelian, the groups $H$ and $K$ are both isomorphic to a quotient of $M$. In other words, the groups $H$ and $K$ are $k$-generated by $M$. So, by Lemma \ref{L_fun}, we have an isomorphism
$$ \phi : L_{H,V} \xrightarrow{\sim} L_{M,L_{H,V}(M)}. $$
In particular, $\phi$ maps the maximal subfunctor $J_{H,V}$ to a maximal subfunctor of $L_{M,L_{H,V}(M)}$. But as explain in Lemma \ref{ev}, the $kB(M,M)$-module $L_{H,V}(M)$ is a projective indecomposable module. Moreover, by hypothesis $M$ is a cyclic group such that $\phi(|M|)$ is invertible in $k$, so the double Burnside algebra $kB(M,M)$ is semi-simple and $L_{H,V}(M)$ is a simple module. In particular, the functor $L_{M,L_{H,V}(M)}$ has a unique maximal subfunctor $J_{M,L_{H,V}(M)}$ which has the property of vanishing at $M$. In conclusion, we have: 
$$J_{H,V}(M)=0. $$
The group $K$ is also isomorphic to a quotient of $M$, so $id_{K}$ factorizes through $M$. In particular, the identity of $J_{H,V}(K)$ factorizes through $J_{H,V}(M)=0$. This implies that $J_{H,V}(K)=0$. In conclusion, we proved that every projective indecomposable functor in $\FDk$ is simple. Any functor in $\FDk$ is a quotient of a direct sum of projective functors. So, a quotient of a direct sum of simple functors. By usual arguments, this implies that every biset functor over $\D$ is semi-simple.  
\item If $H$ and $K$ are two cyclic groups of $\D$, then a $\mathrm{lcm}$ of $H$ and $K$ may not be in the category $\D$. We use the induction and restriction functors between the category $\FDk$ and $\mathcal{F}_{\mathcal{C}yc,k}$ where $\mathcal{C}yc$ is the full subcategory of the biset category consisting of all the cyclic groups. We refer to Section $3.3$ of \cite{bouc_biset} for more details. We use the fact that the projective indecomposable functors of $\FDk$ are exactly the restriction of the projective indecomposable functors of $\mathcal{F}_{\mathcal{C}yc,k}$ indexed by the groups of $\D$. The arguments of the previous point ($\star$) shows that any projective indecomposable functor of $\mathcal{F}_{\mathcal{C}yc,k}$ is simple. Since the restriction functor sends the simple functors indexed by groups of $\D$ to simple functors, the result follows. 
\item[$\star$] In order to apply the previous point we need to show that if $H$ and $K$ are groups in $\D$ and $M$ is a $\mathrm{lcm}$ of $H$ and $K$, then the double Burnside algebra of $M$ is semi-simple. By assumption, we now that $\phi(|H|)$ and $\phi(|K|)$ are invertible in $k$ this implies that $\phi(|M|)=\phi(\mathrm{lcm}(|H|,|K|)$ is invertible in $k$. So, the algebra $kB(M,M)$ is semi-simple. 
\end{itemize}
\end{proof}
In general, when the category of biset functors is not semi-simple, it is possible that some simple functors are also projective. However, the simple functor indexed by the trivial group $1$ is always as far from being projective as possible. For that reason we call it the trivial functor. In the rest of the section, we prove that a category of biset functors is semi-simple if and only if the trivial object is projective. 
\begin{de}\label{trivial}
Let $k$ be a field.
\begin{itemize}
\item Let $\D$ be a replete biset category. The simple functor $S_{1,k}$ of $\FDk$ is called the trivial functor.
\item Let $G$ be a finite group. The simple $kB(G,G)$-module $S_{1,k}(G)$ is called the trivial module.
\end{itemize}
\end{de}
First, we need a technical lemma about the Burnside module. 
\begin{lemma}\label{cyclic}
Let $k$ be a field. Let $G$ be a cyclic $p$-group such that $\mathrm{char}(k) \mid |\Out(G)|$. Then $kB(G)$ is not a simple $kB(G,G)$-module.
\end{lemma}
\begin{proof}
If $\mathrm{char}(k)\neq p$, then Bouc has already classified the composition factors of the $kB(G,G)$-module $kB(G)$ in Paragraph $5.6.9$ of \cite{bouc_biset}. However, Bouc used the idempotents of the Burnside ring $kB(G)$ in co-prime characteristic. So his method cannot be generalized to the case where the characteristic of the field is $p$. 
Let us look more carefully at the action of $kB(G,G)$ on $kB(G)$. Let $H$ be a subgroup of $G\times G$ and $L$ be a subgroup of $G$. Then, the action of the transitive $G$-$G$ biset $(G\times G)/H$ on the transitive $G$-set $G/L$ is given by the Mackey formula (see formula $(\ref{mackey})$). Since $G$ is a commutative group, the action is given by:
\begin{equation}\label{action}
\big((G\times G)/H\big) \cdot G/L = |[p_{2}(H)\backslash G/L]| G/(H\bullet L).
\end{equation}
Here $|[p_{2}(H)\backslash G/L]|$ is the size of a set of representatives of the double cosets $p_{2}(H) \backslash G/L$ and $H\bullet L$ is the subgroup of $G$ defined by:
\begin{equation*}
H\bullet L = \{g\in G\ ;\ \exists\ l\in L \hbox{ with } (g,l)\in H\}.
\end{equation*}
It is clear that $k_1(H) \leqslant H\bullet L \leqslant p_1(H)$.
\begin{enumerate}
\item If $\mathrm{char}(k)\mid (p-1)$, then $p=1$ in $k$. The action of a transitive $G$-$G$-biset $(G\times G)/H$ on a transitive $G$-set $G/L$ is given by $\big((G\times G)/H\big) \cdot G/L = G/(H\bullet L)$. Let us consider $N(G)$ the subspace of $kB(G)$ defined by:
\begin{equation*}
N(G) = \{\sum_{[L\leqslant G]} \lambda_{L}G/L\in kB(G)\ ;\ \sum_{[L\leqslant G]}\lambda_{L}=0\in k \}.
\end{equation*}
Where $[L\leqslant G]$ denotes a set of representatives of the conjugacy classes of the subgroups of $G$. 

It is a $k$-vector space of codimension $1$. Moreover, it is a non zero proper $kB(G,G)$-submodule of $kB(G)$. 
\item If $n>1$ and $\mathrm{char}(k)=p$, then $p^2 \mid |G|$ and $\dim_k kB(G)\geqslant 3$. Formula (\ref{action}) becomes:
\begin{equation*}
\big((G\times G)/H\big) \cdot G/L =\left\{
\begin{array}{cc}
 0 & \hbox{if $L\neq G$ and $p_{2}(H)\neq G$, }    \\
 G/(H\bullet L)&   \hbox{otherwise.}   \\
 \end{array}
\right.
\end{equation*}
Let us consider the subspace $N'(G)$ of $kB(G)$ defined by:
\begin{equation*}
N'(G):=\{ \sum_{[L\leqslant G]} \lambda_{L}G/L\ ;\ \lambda_{G}=0 \hbox{ and } \sum_{[L\leqslant G]} \lambda_{L} =0 \}. 
\end{equation*}
It is a $k$-vector space of codimension $2$ and we claim that it is also a $kB(G,G)$-submodule of $kB(G)$. Indeed if $H$ is a subgroup of $G\times G$ such that $p_{2}(H)\neq G$, then the action of $(G\times G)/H$ on $N'(G)$ is zero. Let $H$ be a subgroup of $G$ such that $p_{2}(H)=G$. Then the action of $(G\times G)/H$ on a transitive $G$-set $G/L$ is given by:
\begin{equation*}
\big((G\times G)/H\big)\cdot G/L = G/(H\bullet L).
\end{equation*}
We need to check if $H\bullet L$ can be equal to the group $G$. Since $H\bullet L$ is a subgroup of $p_{1}(H)$, we can assume that $p_{1}(H)=G$. 
The map which sends $g\in G/(H\bullet L)$ to an element $l(g)\in L$ such that $(g,l(g))\in H$ induces an isomorphism of groups:
\begin{equation*}
(H\bullet L)/k_{1}(H) \cong L/(k_{2}(H)\cap L). 
\end{equation*}
\begin{itemize}
\item Let $H\leqslant G$ be such that $p_{1}(H)=p_{2}(H)=G$ and $k_{2}(H)=G$. Since $p_{1}(H)/k_{1}(H)\cong p_{2}(H)/k_{2}(H)$, this condition implies that $k_{1}(H)=G$. Since $k_{1}(H)\leqslant H\bullet L$, then $H\bullet L = G$ for every subgroup $L$ of $G$. The space $N'(G)$ is therefore stable by the action of such a transitive $G$-$G$-biset.
\item If $k_{2}(H)=k_{1}(H)<G$. As $G$ is a cyclic $p$-group, then either $L\leqslant k_2(H)$ or $k_2(H)\leqslant L$.  If $L\leqslant k_{2}(H)$ then $k_{2}(H)\cap L =L$, therefore $H\bullet L =k_{1}(H)$. If $k_{2}(H)=L$ then $k_{2}(H)\cap L = k_{2}(H)$. So we have $(H\bullet L)/k_{1}(H) = L/k_{2}(H)$, since $k_{2}(H)=k_{1}(H)$, we have $|H\bullet L|=|L|$. In both cases we cannot have $H\bullet L = G$ and $N'(G)$ is a $kB(G,G)$-module. 
\end{itemize}
\end{enumerate}
\end{proof}
As corollary, we have the following useful reformulation of the Theorems of Barker and Bouc.
\begin{theo}\label{semi-simple}
Let $k$ be a field.
\begin{enumerate}
\item Let $\D$ be a replete biset category. Then, the category $\FDk$ is semi-simple if and only if the simple functor $S_{1,k}$ is projective.
\item Let $G$ be a finite group. Then, the double Burnside algebra $kB(G,G)$ is a semi-simple algebra if and only if the simple module $S_{1,k}(G)$ is projective.
\end{enumerate}
\end{theo}
\begin{proof}
By Theorem \ref{semi_simple}, we know that the category $\FDk$ is semi-simple if and only if every group in $\D$ is cyclic and for every $H\in \D$, $\mathrm{char}(k)$ doest not divide $|\Out(H)|$. For the double Burnside algebra,  by Theorem \ref{bouc_ss}, we know that $kB(G,G)$ is semi-simple if and only if $G$ is cyclic and $\mathrm{char}(k)$ does not divide the order of $|\Out(G)|$. So in both cases, it remains to see that if the category is not semi-simple, then the trivial functor (resp. module) is not projective. Or equivalently that its projective cover $kB$ is not simple. 
\begin{itemize}
\item If $G$ is not cyclic then $kB(G)$ is an indecomposable non simple module. Indeed by the proof of Proposition $6.1$ of \cite{bouc_biset}, the kernel of the linearization functor is a non zero proper submodule of $kB(G)$. As consequence, if the category $\D$ contains a non cyclic group, the functor $kB$ is non simple. 
\item if $G$ is cyclic, then by Theorem \ref{ev_equiv} the evaluation at $G$ induces an equivalence of categories between $\FSk$ and $kB(G,G)\Mod$. The group $G$ is a direct product of cyclic groups of prime power order, $G=P_1 \times \cdots \times P_r$. Let us assume that $\mathrm{char}(k)=p$ divides $|\Out(G)|$, then $p \mid |\Out(P_s)|$ for some $s \in \{1,\cdots, r\}.$ By Lemma \ref{cyclic}, the $kB(P_s,P_s)$-module $kB(P_s)$ is not simple. This implies that the functor $kB$ is not simple in $\FSk$ and by using one more times the equivalence of Theorem \ref{ev_equiv}, we have that $kB(G)$ is not simple.  
\item If $\D$ is a replete category containing a cyclic group $G$ such that $\mathrm{char}(k)\mid |\Out(G)|$, then by the previous point, the $kB(G,G)$-module $kB(G)$ is not simple, so the functor $kB$ is not simple. 
\end{itemize}
\end{proof}
\section{Self-injective property of the double Burnside algebra}
As explained in Section \ref{qh}, if the group $G$ is a non-vanishing group and if $k$ is a field of characteristic zero, then the double Burnside algebra is quasi-hereditary. However, some easy computations show that over a field of positive characteristic the double Burnside algebra may have infinite global dimension. Moreover, in the case of $A_5$, there is a self-injective block isomorphic to $\mathbb{C}[X]/(X^2)$ in $\mathbb{C}B(A_5,A_5)$. As consequence, we wonder under which hypothesis on the field or the group, the double Burnside algebra is a self-injective algebra.
\newline\indent If $\D$ is a replete biset category containing only finitely many isomorphism classes of objects, then by Morita's Theorem the category $\FDk$ is equivalent to the category of modules over a finite dimensional algebra. In particular, we will say that the category $\FDk$ is self-injective if the corresponding finite dimensional algebra is self-injective. Then we have to following result.
\begin{prop}\label{selfinj}
Let $k$ be a field and $\D$ be a replete biset category with only finitely many isomorphism classes of objects. Then $\FDk$ is self-injective if and only if it is semi-simple.
\end{prop} 
\begin{proof}
We only need to prove that if this category is self-injective, it is semi-simple.  If the category of biset functors over $\D$ is self-injective, the application sending the top of a projective indecomposable functor to its simple socle induces a bijection on the set of isomorphism classes of simple functors. This bijection is called the Nakayama's permutation (see Lemma $1.10.31$ of \cite{zimmermann_rep} for more details). In particular, the simple functor $S_{1,k}$ must be in the socle of a projective indecomposable functor. 
\newline \indent Let $P_{H,V}$ be a projective cover of the simple functor $S_{H,V}$. By Theorem $6.3$ of \cite{webb_strat2}, there is a filtration
\[0=P_0 \subset P_1\subset \cdots \subset P_n=P_{H,V}, \]
such that $P_i/P_{i-1} \cong \Delta_{H_i,U_i}$, where $H_i \in \D$ and $U_i$ is a direct summand of a permutation $k\Out(H_i)$-module. So we have: 
\[ \mathrm{Soc}(P_{H,V}) \subseteq \bigoplus \mathrm{Soc}(\Delta_{H_i,U_i}), \]
where the $\Delta'$s runs through the standard quotients of $P_{H,V}$. In particular, if the simple functor $S_{1,k}$ is in the socle of $P_{H,V}$, then it is in the socle of some of its standard factors $\Delta_{H_i,U_i}$. Also, for a finite group $K$, if $\Delta_{H_i,U_i}(K)\neq 0$, then $H_i$ is a subquotient of $K$. So if $S_{1,k}$ is composition factor of $\Delta_{H_i,U_i}$ then $H_i=1$ and $U_i \cong \oplus k$. Moreover, $\Delta_{1,k} = L_{1,k}=kB$. As consequence, the simple functor $S_{1,k}$ only appears at the top of $\Delta_{1,k}$. So if the simple functor $S_{1,k}$ is in the socle of $P_{H,V}$ then it is in the socle of $\Delta_{1,k}$. This is the case if and only if $\Delta_{1,k}$ is simple. By Theorem \ref{semi-simple} this is the case if and only if $\FDk$ is semi-simple. 
\end{proof}
Now, we state the result for the double Burnside algebras. 
\begin{theo}
Let $k$ be a field. Let $G$ be a finite group. Then, the double Burnside algebra $kB(G,G)$ is self-injective if and only if it is semi-simple. 
\end{theo}
\begin{proof}
Since the group $1$ is a quotient of $G$, the evaluation at $G$ of $S_{1,k}$ is a non-zero simple module. Moreover, by Corollary \ref{lift_proj}, the simple $kB(G,G)$-modules are the non-zero evaluation of the simple functors. And if $S_{H,V}(G)\neq 0$, then $P_{H,V}(G)$ is a projective cover of this simple module. By Theorem $6.3$ of \cite{webb_strat2}, the projective functor $P_{H,V}$ has a filtration 
\[0=P_0 \subset P_1\subset \cdots \subset P_n=P_{H,V}, \]
such that $P_i/P_{i-1} \cong \Delta_{H_i,U_i}$, where $H_i \in \D$ and $U_i$ is a direct summand of a permutation $k\Out(H_i)$-module.
Since the evaluation at $G$ is an exact functor, the $kB(G,G)$-module has a (weak form of) filtration:
\[0=P_0(G) \subseteq P_{1}(G)\subseteq \cdots \subseteq P_n(G)=P_{H,V}(G), \]
and the quotients of this filtration are:
\[P_i(G)/P_{i-1}(G) \cong \Big(P_i/P_{i-1}\Big)(G) \cong \Delta_{H_i,U_i}(G).\]
Note that some of these factors may be zero, but not all of them since $S_{H,V}(G)\neq 0$ by hypothesis. Moreover, we have that $\mathrm{Soc}\big(P_{H,V}(G)\big) \subseteq \bigoplus\Big(\mathrm{Soc}\Delta_{H_i,U_i}(G)\Big).$ 
\newline\indent Let us assume that $kB(G,G)$ is a self-injective algebra. Then the simple module $S_{1,k}(G)$ must be in the socle of a projective indecomposable module $P_{H,V}(G)$. So $S_{1,k}(G)$ is in the socle of some $\Delta_{H_i,U_i}(G)$. Moreover, by Proposition \ref{ev_fact}, $S_{1,k}(G)$ is composition factor of $\Delta_{H_i,U_i}(G)$ if and only if $S_{1,k}$ is composition factor of $\Delta_{H_i,U_i}$ in the category of biset functors. So by the Proof of proposition \ref{selfinj}, we have $H_i=1$ and $U_i \cong \oplus k$. So $S_{1,k}(G)$ must be in the socle of $\Delta_{1,k}(G)=kB(G)$. By Theorem \ref{semi-simple}, this implies that $kB(G,G)$ is a semi-simple algebra. 
\end{proof}
\begin{re}
As corollary, we have that the double Burnside algebra of a finite group $G$ over a field $k$ is symmetric if and only if it is a semi-simple algebra. In \cite{trace_maps}, the author studied the symmetry of the Mackey algebra. The main tool was a central linear map on the Mackey algebra which comes from the monoidal structure of the category of modules over the Mackey algebra, that is the category of Mackey functors. There are lot of points in common between the theory of biset functors and the theory of Mackey functors. In particular, the category of biset functors is also a closed symmetric monoidal category under suitable hypothesis on the category $\D$. (see Chapter 8 of \cite{bouc_biset}). The trace map of this monoidal structure (see Section 4 of \cite{may_picard} for more details about the trace of a monoidal category) is a map which goes from the endomorphism ring of a finitely generated projective biset functor to the endomorphism ring of the Burnside functor. By taking a representable functor $Y_G$ , we have a central linear map:
\begin{align*}
tr : kB(G,G) \to End_{\FSk}(Y_G) \cong kB(1) \cong k.
\end{align*}
One can compute this trace and show that if $U$ is a $G$-$G$-biset then $tr(U) = |U/G|\in k$, that is the number of $G$-orbits in $U$ where $G$ acts diagonally on $U$. Unfortunately, this map cannot help to the comprehension of the symmetry of $kB(G,G)$ since the bilinear form $(U, V )\mapsto tr(U \times_G V)$ is always degenerate when $G \neq 1$.
\end{re}

\bibliographystyle{abbrv}

\end{document}